\documentclass{egpubl}
\usepackage{egsgp2026}
\usepackage{caption}
\usepackage{wrapfig}

\usepackage{amsthm}
\usepackage[T1]{fontenc}
\usepackage{dfadobe}  
\usepackage[linesnumbered,ruled,vlined]{algorithm2e}

\usepackage{cite}  
\BibtexOrBiblatex  
\electronicVersion
\PrintedOrElectronic
\ifpdf \usepackage[pdftex]{graphicx} \pdfcompresslevel=9
\else \usepackage[dvips]{graphicx} \fi

\usepackage{egweblnk} 

\usepackage{caption, subcaption}
\usepackage{amsmath, amssymb}
\usepackage{booktabs}
\usepackage{multirow}
\usepackage{siunitx}

\newcommand{\ind}[1]{\mathrm{ind}(#1)}

\newtheorem{theorem}{Theorem}  
\newtheorem{definition}{Definition}
\newtheorem{lemma}{Lemma}      

\newcommand{\noam}[1]{}

\newcommand{\mesh}{\mathcal{M}}

\newcommand{\z}{\textbf{z}}
\newcommand{\dir}{\vec{r}}

\newcommand{\n}{\textbf{n}}
\newcommand{\e}{\textbf{e}}
\newcommand{\twine}{\text{\reflectbox{\e}}}

\newcommand{\vv}{\textbf{u}}
\newcommand{\flux}{\mathcal{F}}

\newcommand{\lap}{\textbf{L}}
\newcommand{\real}{\mathbb{R}}

\newcommand{\tri}{\textbf{t}} 

\newcommand{\sign}[1]{\text{sign}\parr{#1}}
\newcommand{\norm}[1]{\left\lVert#1\right\rVert}
\newcommand{\brac}[1]{\left[#1\right]}
\newcommand{\tribrac}[1]{\left<#1\right>}

\newcommand{\parr}[1]{\left(#1\right)}
\newcommand{\set}[1]{\left\{#1\right\}}

\newtheorem{property}{Property}

\title[A Tutte-Like Theorem for Injective Flows on Meshes]%
      {A Tutte-Like Theorem for Injective Flows on Meshes}

\author[D. Kaushik \& N. Aigerman]
{\parbox{\textwidth}{\centering D. Kaushik$^{1,2}$ and N. Aigerman$^{1,2}$}
        \\
{\parbox{\textwidth}{\centering $^1$Université de Montréal \\ $^2$Mila - Quebec Artificial Intelligence Institute}
}
}

\begin{document}

\maketitle
\thispagestyle{plain} 
\pagestyle{plain}     

\begin{abstract}
This work proves a theorem regarding injectivity of continuous flows that are defined via discrete vector fields on triangle meshes. In similar fashion to Tutte's embedding, the theorem shows a connection between harmonicity and injectivity. Specifically,  a discrete harmonic function is computed over a spherical mesh, and its $90^\circ$-rotated gradients are taken as the vector field. We show that this leads to an injective flow over the mesh, i.e., when two points are advected along the vector field, their trajectories are guaranteed to not intersect. The core insight behind this theorem is that the characterization of the Poincaré-Hopf indices of discrete harmonic 1-forms is paralleled in the combinatorial rules for injective vector fields. Similarly to Tutte's embedding, the injective vector field can be computed by simply solving a sparse linear system. Furthermore, it is differentiably dependent on the entries of the Laplacian, enabling unconstrained gradient-based optimization of the vector field while preserving injectivity. The advection itself can be computed by a quick and accurate algorithm. We experimentally validate the injectivity of the results, as well as show examples of differentiable optimization, such as computing an injective approximation to a given non-injective vector field, and fitting an injective vector field to satisfy constraints.
\end{abstract}  

\section{Introduction}

This paper proposes a method for producing discrete, triangle-based tangent vector fields that define injective flows over triangle meshes, i.e., advecting points along the vector field yields continuous trajectories for any given point on the surface,  such that the trajectories  of any two points do not cross. Figure~\ref{fig:crab} shows such an advection of randomly sampled points on a mesh, their trajectories traced behind them as faded paths.

The mathematical challenge in ensuring the induced flow is injective lies in the discrete nature of the vector field. Indeed, when considering a \emph{continuous} vector field, to ensure injectivity it is sufficient to require the vector field to be Lipschitz  (per the Picard–Lindelöf theorem). In contrast, discrete, face-based vector fields are piecewise-constant and \textit{discontinuous}; thus, their injectivity is a much more brittle property, which requires a different, more discrete approach, such as the satisfaction of combinatorial, triangle-based rules~\cite{campen2016}.



Our main theoretical contribution --  Theorem~\ref{thm:injectivity} -- is akin to a flow-based version of Tutte's embedding~\cite{tutte1963,floater1997parametrization}, proving that a specific class of \textit{discrete harmonic functions} induces injective flows. Specifically, we prove that solving a discrete Laplace equation on a spherical mesh with specific boundary conditions yields a discrete harmonic function whose gradient vectors, rotated by $90^\circ$, define an injective flow. Geometrically, this implies that each point's trajectory lies on a fixed level set of the harmonic function. Similarly to Tutte's embedding, these vector fields can be computed in a simple and efficient manner by solving a sparse linear system.

\begin{figure}[t]
    \centering
    \includegraphics[width=\linewidth]{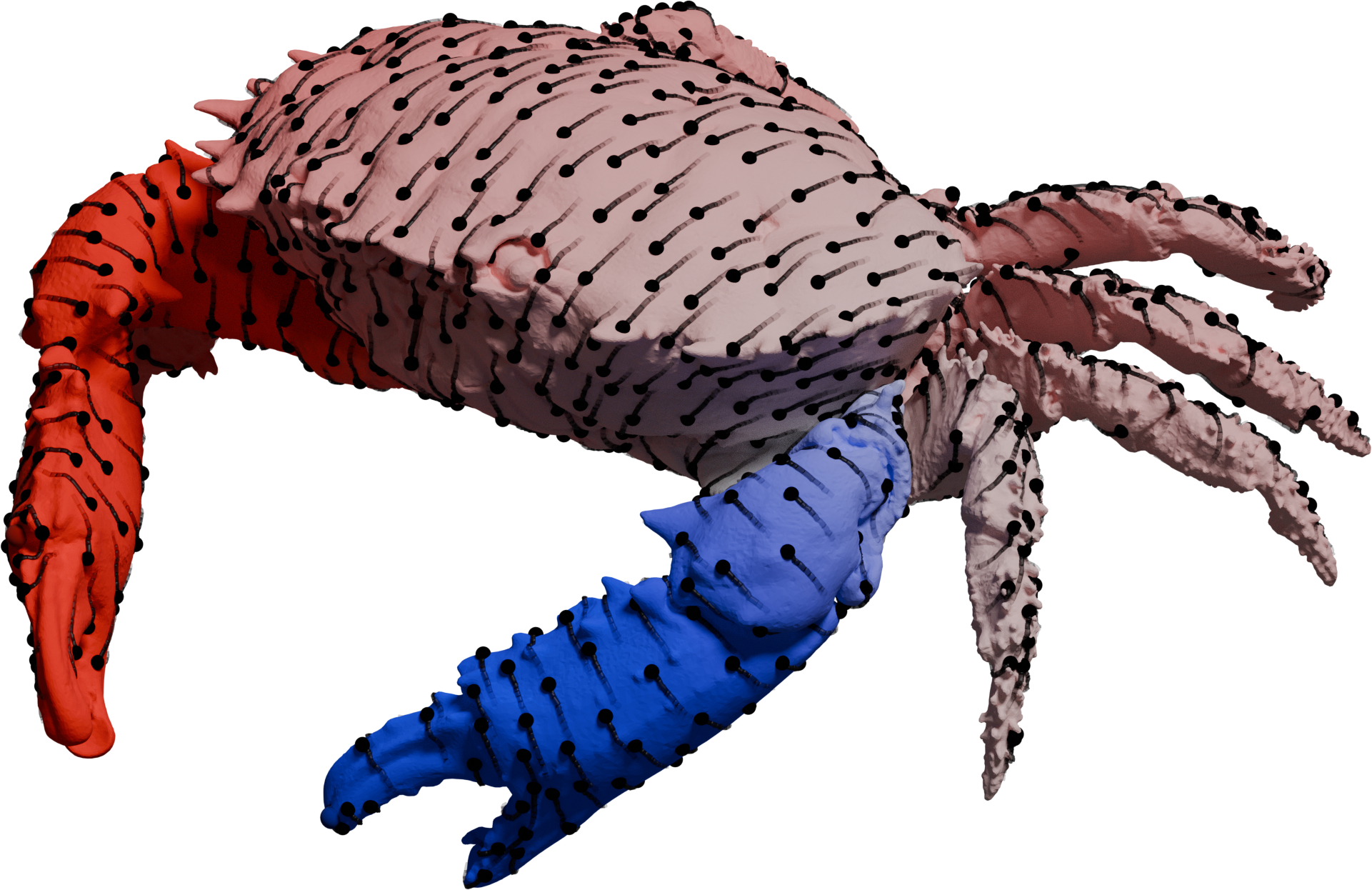}
    \caption{Advecting points (in black) on a mesh, with respect to a harmonic function (visualized via colors). Our main result guarantees that this flow is injective: no two trajectories intersect.}
    \label{fig:crab}
\end{figure}

\begin{figure*}[t]
     \centering
     \begin{subfigure}[t]{0.49\textwidth}
         \includegraphics[width=\textwidth]{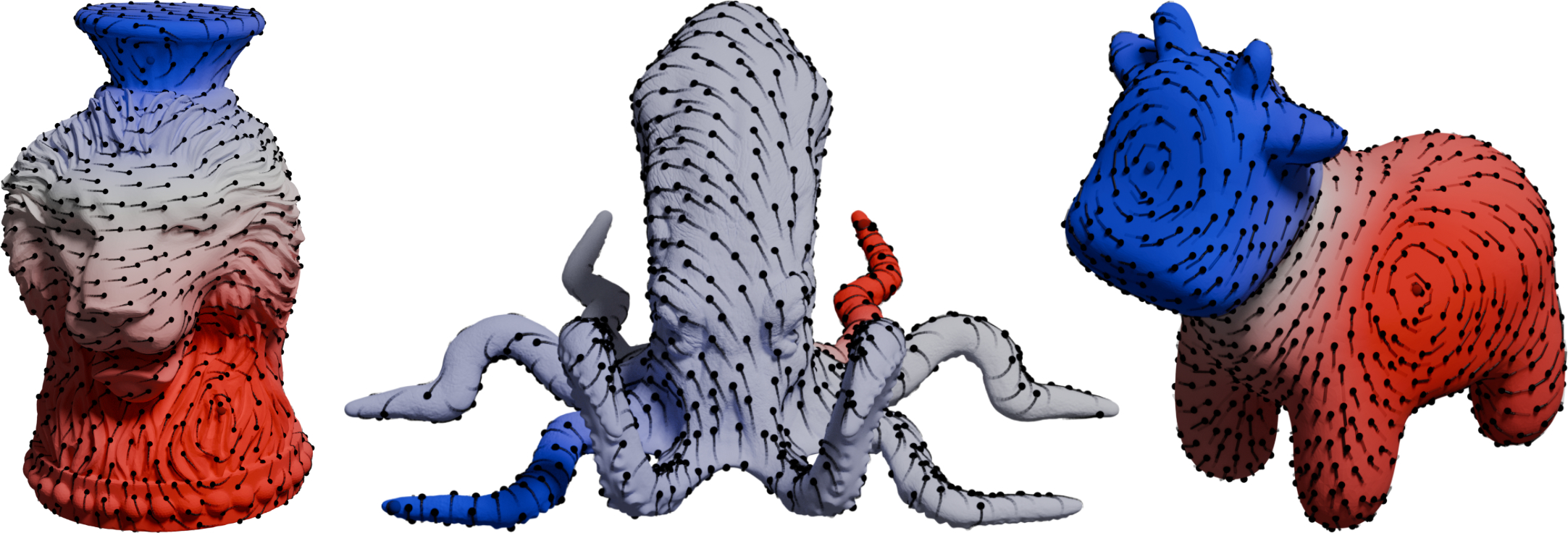}
         \caption{Cotangent Laplacian}
     \end{subfigure}
     \begin{subfigure}[t]{0.49\textwidth}
         \includegraphics[width=\textwidth]{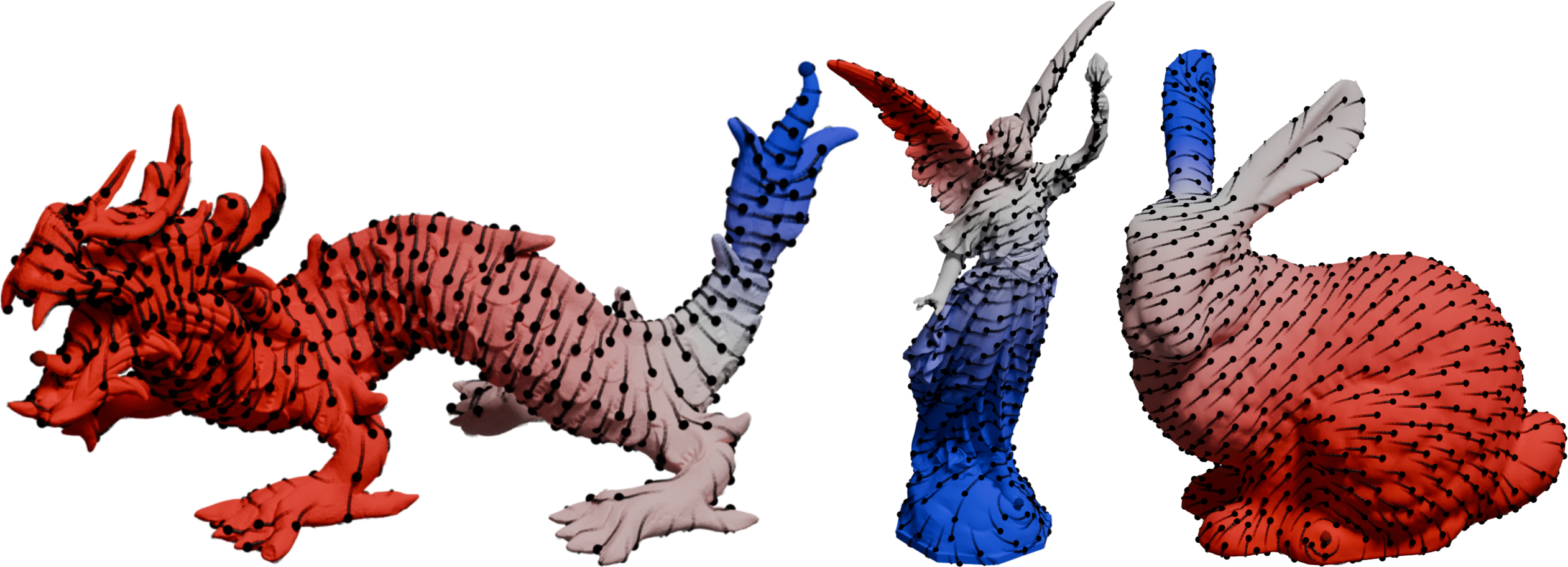}
         \caption{Mean Value Coordinates Laplacian}
         \label{subfig: mvc init}
     \end{subfigure}
     
    \caption{\textbf{Injective flows defined by standard Laplacians used in geometry processing. } Any choice of Laplacian with positive edge weights yields a different discrete harmonic function, which is guaranteed by Theorem~\ref{thm:injectivity} to result in an injective flow. We show results for the cotangent Laplacian~\cite{pinkall1993} and the mean value coordinates (MVC)~\cite{floater2003mean} Laplacian.}
    \label{fig:cot_and_mvc}
\end{figure*}

Theorem~\ref{thm:injectivity} is attained through two main insights (Lemmas~\ref{lem:flux_1form_equiv} and \ref{lem:us_gortler_equiv}) regarding a connection between two rather different approaches for injective mappings of triangle meshes: the framework of \textit{bijective foliations}~\cite{campen2016} defines conditions on vector fields to yield injective paths on meshes; however, these conditions are combinatorial, and hence were previously attained by constructing these vector fields in a combinatorial, algorithmic manner. We observe that these conditions for injectivity are exactly reflected in harmonic functions' specific properties (namely, a discrete Poincaré-Hopf index theorem) that were characterized by Gortler et al.~\cite{gortler2006} -- see Figure~\ref{fig:vertex_inj}, which exhibits the equivalence between the two approaches. This observation, in turn, enables the construction of  vector fields that satisfy the injectivity conditions from~\cite{campen2016} via harmonic functions.

While we consider the proof  our main technical contribution, it holds a secondary practical implication of a differentiable parameterization of injective flows on meshes: any assignment of positive edge weights of the Laplacian matrix modifies the Laplace equation, and hence the resulting vector field and flow, in a differentiable manner. Furthermore, our  discrete vector field defines a flow that is piecewise-linear on the triangles of the mesh, leading to simple, efficient and accurate computation,  in contrast to techniques that define flows via non-discrete, continuous vector fields~\cite{chen2018neural,chen2023riemannianfm}; in order to compute their corresponding flow, one  needs to integrate the vector field via quadrature rules, which is both expensive and may lead to inaccuracies. 

To verify our theory, we validate the injectivity of our vector field both qualitatively and quantitatively. We additionally show preliminary demonstrations of the utility of a differentiable representation of injective vector fields, by optimizing the injective field to best approximate a given non-injective field, or fitting it to satisfy constraints.


    

\section{Related Works}

\paragraph*{Vector field design.}
Vector field design has been heavily studied in computer graphics and geometry processing~\cite{zhang2006vector,degoes2016, vaxman2016}. Similarly to our approach, previous papers have considered vector fields as gradients of functions~\cite{diamanti2015integrable} and sought to integrate them~\cite{ray2006periodic}.  This integration is similar to a flow and is especially common in the context of global parameterization~\cite{myles2014robust} and quad-meshing~\cite{bommes2013quad}. In these applications, the cases leading to non-injectivity of the \emph{flow} simply lead to singularities in the parameterization/quad mesh (and \emph{not} to loss of injectivity of, e.g., the parameterization) -- these are in general acceptable, hence those methods usually focus on minimizing the number of such singularities. Many works focus solely on unit-length vector fields~\cite{solomon2019optimal,boksebeld2022high,verhoeven2022dev2pq,palmer2024lifting,wang2025power} and their extensions to polyvector fields~\cite{palacios2007rotational,bommes2010practical,diamanti2014designing,panozzo2014frame} in which the flow direction can be chosen out of several options at each triangle.
Another  common goal is optimizing the smoothness of the vector/directional field~\cite{fisher2007design,knoppel2013globally}, while ensuring adherence to singularity assignment~\cite{crane2010}; others optimize for vector fields which produce an as-rigid-as-possible flow~\cite{benchen2010killing}. Flows on vector fields additionally have a close connection to geodesics~\cite{sharp2019vector,vekhter2019weaving}.

\paragraph*{Injective maps between meshes.}
The flows we compute define a map between a mesh and itself; defining such injective mappings of meshes is a highly researched area. Most previous approaches use a \emph{common domain} into which the meshes are mapped, thereby defining a mapping by composition of maps into the common domain~\cite{kraevoy2004,aigerman2014lifted,aigerman2015seamless,aigerman2016hyperbolic,morreale2021,schmidt2020inter}. Of note is one of the first works on mesh-mesh injective mappings~\cite{schreiner2004inter}, which directly deals with injectively sliding the vertices of one mesh over the triangles of the other in a sort of ad hoc flow. Bijective maps between domains can also be defined via integral curves~\cite{huynh2015}, global structural decompositions~\cite{nieser2011, bommes2013, hinderink2024}, or by auxiliary structures~\cite{jiang2020, jiang2021, liu2024}. 
Lastly, simplicial foliations~\cite{campen2016, hinderink2021} can be used for defining vector-aligned slices on surfaces and volumes, and then matching between corresponding slices on two meshes. These slices are in turn created by an injective vector field -- we directly use the conditions formulated by~\cite{campen2016} and show they hold for specific harmonic functions.

\paragraph*{Harmonic functions and injective mappings.}
There is a deep  connection between harmonic functions over 2-manifolds and injective mappings, with the most central results  being the Radó–Kneser–Choquet theorem and its discrete counterpart, Tutte's embedding~\cite{tutte1963,floater1997parametrization,floater2003one}. Tutte's embedding states that computing two discrete harmonic functions representing the $x,y$ coordinates of a disk-topology mesh, while prescribing the boundary conditions to a convex polygon, yields an injective embedding of the mesh to 2D. This has been extended further into other topologies~\cite{gortler2006,aigerman2015orbifold} and non-Euclidean target domains~\cite{tsui2013globally, aigerman2016hyperbolic,aigerman2017spherical}. Other methods use more complex harmonic structures for injectivity, which require more careful optimization~\cite{bright2017harmonic,hefetz2019subspace}. 
Previous works have considered the use of  harmonic functions and flowing via their gradients for quad-meshing~\cite{dong2005harmonic}, without considerations of injectivity, and for injective parameterization, but not flows~\cite{tong2006designing}. Some use harmonic gradients for constructing maps; however, the injectivity of these is guaranteed only in  the smooth setting, relying on the harmonic maps' infinite smoothness~\cite{lin2015interior} and the existence of hessians~\cite{xia2010direct}. In contrast, we prove a discrete theorem, showing that injectivity holds directly on the discrete gradients of the discrete harmonic function, for \emph{any} Laplacian matrix. We further leverage the fact that discrete harmonic functions are defined with respect to any Laplacian matrix defined with positive weights on the mesh's edges, and optimize the  injective vector field via the Laplacian's entries. Such optimization of the Laplacian in order to control injective maps has been performed before~\cite{aigerman2024generative,sun2024tuttenet}.

\paragraph*{Computing continuous flows.}
Computing \emph{continuous} injective flows  is a highly active area, especially in machine learning where it is used for generative techniques~\cite{lipman2023}.  These flows can be computed in a differentiable process via a neural network that produces a continuous vector field  on Euclidean spaces~\cite{chen2018neural} and on manifolds~\cite{falorsi2020, lou2020, mathieu2020,bose2020, rozen2021, chen2024,benhamu2022}. Some previous works have used such flows to define deformations of 3D shapes~\cite{jiang2020shapeflow}. Injectivity is guaranteed as long as the vector field is Lipschitz. In contrast, we focus on piecewise-constant vector fields defined on a piecewise-flat triangle mesh, and on ensuring they still define an injective flow.

\section{Problem Setting and Goal}

\subsection{Preliminaries}
We consider sphere-topology manifold triangle meshes $\mesh$. 
We will focus on working with oriented \textit{half-edges} $\e=(i,j)$, where, for a triangle $(i,j,k)$, its oriented half-edges follow the triangle's orientation: $(i,j),(j,k),(k,i)$. Note that each half-edge belongs to exactly one triangle. The \textit{twin} half-edge of $\e=(i,j)$ is denoted $\twine=(j,i)$.  Each such half-edge has its \emph{edge normal}, defined as the vector obtained by inverting the unit-normalized directed edge vector and rotating it by $90^\circ$ in the tangent plane of its triangle:
\begin{equation}
	\n_{\e}^\tri=-\textbf{R}^{90^\circ}_\tri \parr{v_j-v_i}/\norm{v_i-v_j}, 
\end{equation}
where $\textbf{R}^{90^\circ}_\tri$ is the matrix that rotates vectors in the tangent plane of triangle $\tri$ by $90^\circ$. 
Note that this normal points inward into the triangle, towards the vertex $k$ opposite to the edge $(i,j)$. 

We will consider \textit{discrete functions}, i.e., an assignment of a scalar $z_i$ to each vertex $v_i$ of the mesh. We will also use the mesh's \emph{gradient} operator $\nabla$, defined via linear finite elements, assigning to each triangle the gradient vector of the affine function interpolating between the three values on the vertices, denoted $\nabla_\tri \z$. 

In standard fashion, we assign a piecewise-constant vector field $\vv$ over the faces of the mesh, defined by assigning a tangent vector $\vv_\tri\in T_\tri$ to each triangle $\tri$ of the mesh (see Figure~\ref{fig:flux}). Each half-edge $\e$ thus has exactly one corresponding tangent vector, belonging to the triangle $\tri$ that contains $\e$.

A \emph{flow} over the mesh is defined for any point $p\in\mesh$ (not necessarily a vertex), assigning it a continuous path $p_t$ parameterized by a time parameter $t\in\brac{0,1}$ such that the flow starts from $p$, i.e., $p_0 = p$, and advects that point to somewhere else. The derivative with respect to time $t$ is denoted $\frac{\partial p_t}{\partial t}$.
We consider flows with respect to a vector field $\vv$ such that the flow ``agrees'' with the vector field inside each triangle, namely, if at time $t^*$ the point $p_{t^*}$ is  strictly inside triangle $\tri$ -- i.e., away from the triangle's edges --  then the speed vector of the point is equal to the vector field assigned to that triangle,
\begin{equation}
    \frac{\partial p_t}{\partial t}\parr{t^*} = \vv_\tri \text{\ \ \ \ \ \            (if $p_{t^*}\in \tri$)}.
\end{equation}

\begin{figure}[t]
    \centering
    \includegraphics[width=\linewidth]{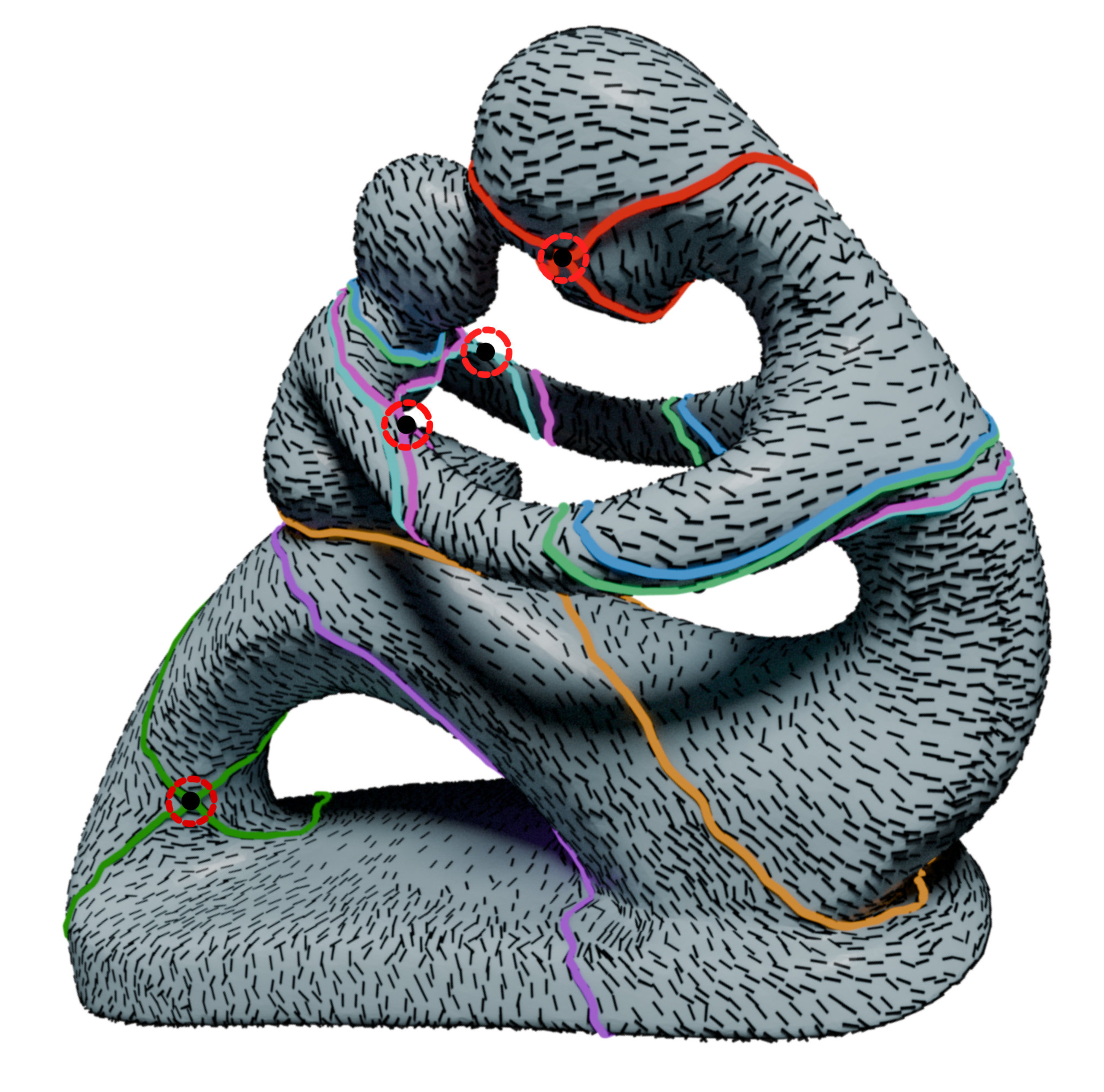}
    \caption{\textbf{Non-injective anti-example.} A discrete harmonic function does not necessarily define an injective flow, e.g., when applied to a non-spherical topology. Trajectories visualized with different colors. Intersections where flows from multiple directions cross are highlighted with red circles.}
    \label{fig:non-injective-isolines}
\end{figure}

In other words, the vector field defines a flow inside each triangle  by advecting each point using the vector field as the temporal derivative, which for a constant vector field reduces to a linear formula:
\begin{equation}
    p_t = p_0+t \vv_\tri.
\end{equation} 
Since the vector field is piecewise-constant, the flow is piecewise-linear, as the inset illustrates.
\setlength{\columnsep}{5pt}
\begin{wrapfigure}[3]{r}{0.2\columnwidth}
   \centering
   \vspace{-14pt}
   \includegraphics[width=\linewidth]{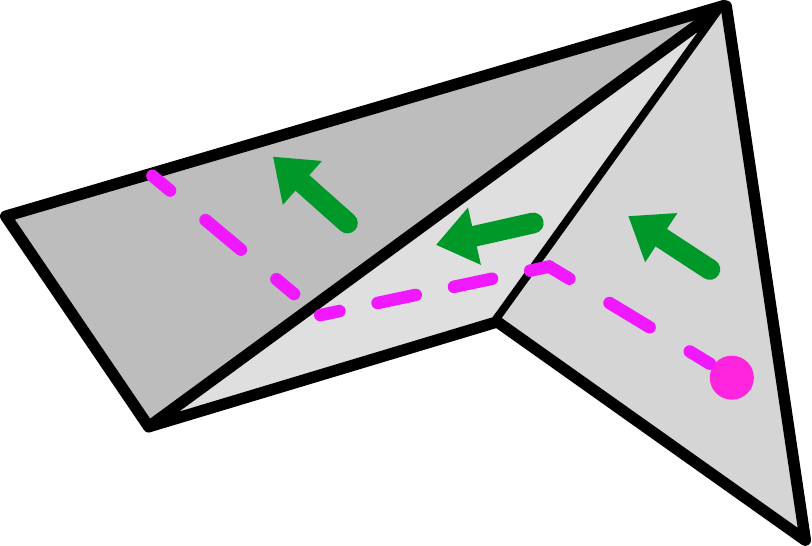}
\end{wrapfigure}
The question of how to well-define the flow when the point crosses over between triangles (over an edge or a vertex) is equivalent to ensuring the flow is injective. Figure~\ref{fig:flux} shows examples of flows, where the purple point is flowed to its end position marked by a red $X$.

 To tackle injectivity, we will show that a harmonic function's gradient field satisfies the combinatorial rules for injectivity from~\cite{campen2016}. We next review these rules, and  reformulate them into a formulation more suited to our needs, using  the notion of edge/vertex fluxes, illustrated in Figure~\ref{fig:flux}.

\subsection{Fluxes }
\begin{figure}[t]
     \centering
     \begin{subfigure}[b]{0.32\columnwidth}
         \includegraphics[width=\textwidth]{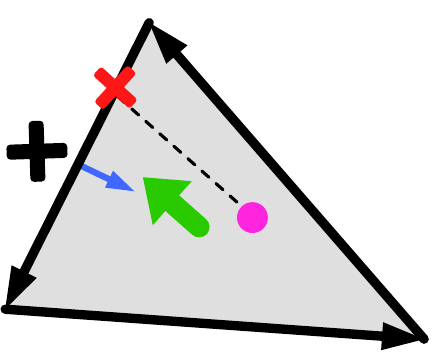}
         \captionsetup{justification=centering}
         \caption*{\textbf{Positive} edge flux}
     \end{subfigure}
     \begin{subfigure}[b]{0.32\columnwidth}
         \includegraphics[width=\textwidth]{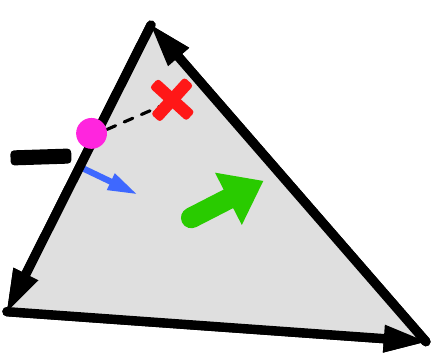}
         \captionsetup{justification=centering}
         \caption*{\textbf{Negative} edge flux}
     \end{subfigure}\\
     \begin{subfigure}[b]{0.32\columnwidth}
         \includegraphics[width=\textwidth]{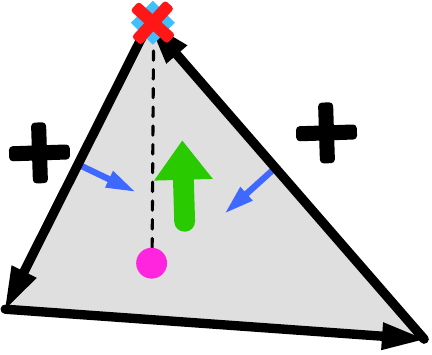}
         \captionsetup{justification=centering}
         \caption*{\textbf{Positive} vertex flux\\\textbf{Even} parity ($\textbf{+}=\textbf{+}$)}
     \end{subfigure}
     \hfill 
     \begin{subfigure}[b]{0.32\columnwidth}
         \includegraphics[width=\textwidth]{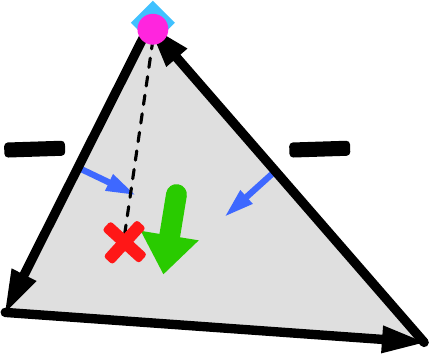}
         \captionsetup{justification=centering}
         \caption*{\textbf{Negative} vertex flux\\\textbf{Even} parity ($\textbf{--}=\textbf{--}$)}
     \end{subfigure}
     \hfill
     \begin{subfigure}[b]{0.32\columnwidth}
         \includegraphics[width=\textwidth]{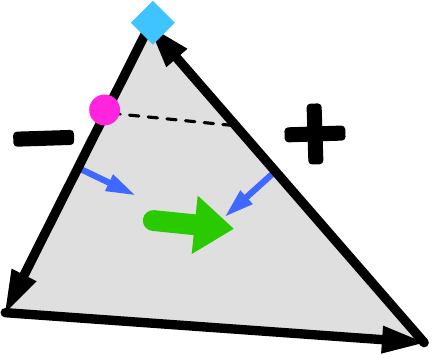}
         \captionsetup{justification=centering}
         \caption*{\textbf{Zero} vertex flux\\\textbf{Odd} parity ($\textbf{--}\neq\textbf{+}$)}
     \end{subfigure}
     
    \caption{\textbf{Types of fluxes on edges and vertices.} The three different types of vertex fluxes, and the parity of each one, along with the two types of edge fluxes. Vector field in green, edge normals in blue, flowed point in purple, its trajectory as a dashed line, and its final position as red X. Edge flux sign denoted with +/- next to the edge.}
    \label{fig:flux}
\end{figure}
We begin by discussing fluxes over \textit{edges} (which will be subsequently used to also define fluxes over \textit{vertices}). We say the half-edge $\e$ has \emph{positive flux} if there exists a point inside its corresponding triangle, $p\in\tri$, such that the flow defined by $\vv_\tri$ will flow this point to lie on the edge $\e$.  Inversely, we say the vector has \textit{negative} flux with respect to the edge $\e$ if there is a point $p$ on the edge $\e$ which would flow into the triangle.

  We denote the edge flux on the half-edge $\e$ as $\flux_\e\in\set{+1,-1}$. Geometrically, it corresponds to the inverted sign of the dot product of the vector with the edge normal: 
  \vspace{5pt}
  \begin{definition}[Edge flux]
\label{def:edge_flux}
The edge flux $\flux_\e$ on edge $\e$ is defined as
\begin{equation}
\flux_\e=-\sign{\tribrac{\vv_\tri,\n_\e}}.
\end{equation}
\end{definition}
See Figure~\ref{fig:flux}, top row, where a $+$/$-$ sign denotes the flux over the corresponding half-edge, the purple point is the starting point $p$ and the red X marks the final point.
Note that we assume we do not encounter the case of a vector field parallel to the edge (no point is flowed into or out of that edge, and the inner product is zero).

We now move on to discuss fluxes on \emph{vertices}. Consider a triangle $\tri$, the corresponding vector on the triangle $\vv_\tri$, and one of the  vertices $v$ of that triangle. We say the vector has \emph{positive flux} with respect to the vertex $v$ if there exists a point on the triangle, $p\in\tri$, such that the flow defined by $\vv_\tri$ will flow this point to lie on the vertex $v$.  We say the vector has \textit{negative} flux with respect to the vertex $v$ if a point starting from the vertex flows  into the triangle. Finally, we say there is \emph{no flux} in all other cases, which essentially correspond to the vertex neither flowing inside nor outside of this triangle.

To formalize these conditions geometrically: note that  the vertex lies on the intersection of two of the triangle's edges, $e_1,e_2$. Therefore, for a point $p$ to flow to the vertex, it must flow into both edges, hence the vector $\vv_\tri$ must have positive edge flux into \emph{both} edges sharing that vertex. Similarly, for the vertex to flow into the triangle, the vector $\vv_\tri$ must have negative flux into both edges sharing the vertex. In the remaining case (no vertex flux), the vector has positive flux on one edge, and negative flux on the other. See Figure~\ref{fig:flux}, bottom for an illustration of these cases.
\vspace{5pt}
\begin{definition}[Vertex flux]
\label{def:vertex_flux}
The flux on vertex $v$ with respect to triangle $\tri$ is  defined via the edge fluxes as:
\begin{equation}
\label{eq:vertex_flux}
    \flux^\tri_{v} = 
\begin{cases} 
+1 & \text{if }\flux_{(v,i)} = +1 \text{ and } \flux_{(j,v)} =+ 1 \\
-1 & \text{if }\flux_{(v,i)} = -1 \text{ and } \flux_{(j,v)} =-1 \\
\ \ \ 0 &   \text{if }\flux_{(v,i)} = -\flux_{(j,v)} 
\end{cases}
\end{equation}
\end{definition}

\subsection{Injectivity Conditions, Defined via Fluxes}
\label{ss:injective_flux}
\begin{figure}[t]
     \centering
     \begin{subfigure}[b]{0.32\columnwidth}
         \includegraphics[width=\textwidth]{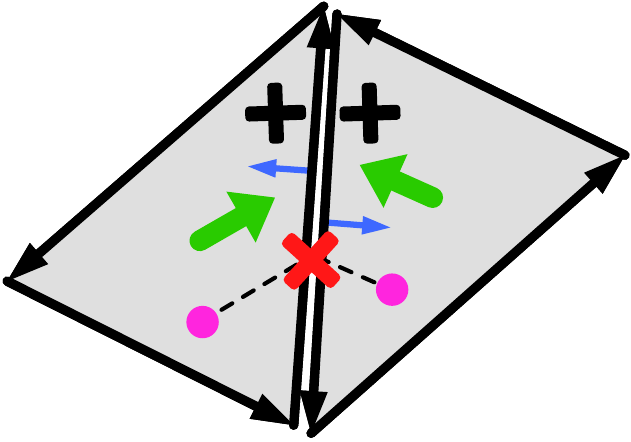}
         \captionsetup{justification=centering}
         \caption*{Non-injective edge flux}
     \end{subfigure}
     \begin{subfigure}[b]{0.32\columnwidth}
         \includegraphics[width=\textwidth]{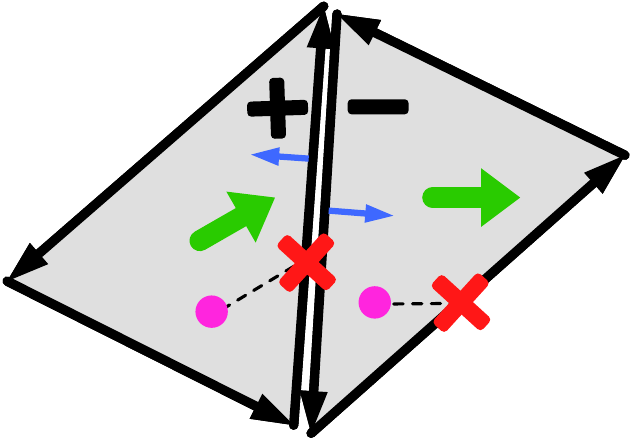}
         \captionsetup{justification=centering}
         \caption*{Injective edge flux}
     \end{subfigure}
     \hfill 
     \begin{subfigure}[b]{0.32\columnwidth}
     \includegraphics[width=\linewidth]{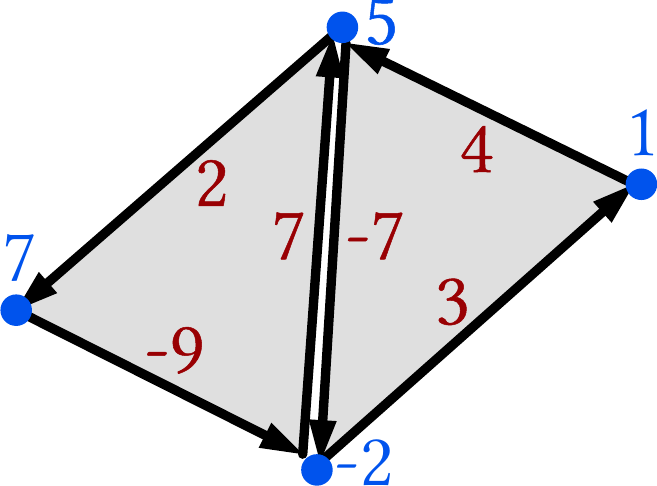}
     \caption*{{\color[RGB]{0, 83, 237}Function $\z$} and {\color[RGB]{153, 0, 3} 1-form $\delta$}}
     \end{subfigure}
     \caption{\textbf{Injectivity of edge fluxes.} Injectivity dictates that no two points are advected to the same location on the edge. This  is equivalent to requiring opposite flux signs between the half-edge and its twin. The 1-form $\delta$ exactly satisfies this antisymmetric relation. }
     \label{fig:edge_inj}
\end{figure}
Any vector field well-defines a flow in the \emph{interior} of the triangle, but not necessarily on the \emph{boundary} of the triangle: two points on two triangles may flow into the same point either on an edge or on a vertex. This is equivalent to non-injectivity. \cite{campen2016} define specific conditions that ensure injectivity, which we review next while reformulating them for our needs.


First, it is immediate to observe (e.g., via Figure~\ref{fig:edge_inj}, left and center) that the vector field is injective over the half-edges $\e$  and $\twine$ iff $\e$ has flux (per Definition~\ref{def:edge_flux}) with a sign opposite to that of $\twine$, 
\begin{equation}
\label{eq:edge_inj}
\flux_\e = - \flux_{\twine}
\end{equation}
The flow of a point $p$ is then defined  by the vector with the negative flux, which flows  points away from the edge. \vspace{5pt}
\begin{property}[edge injectivity]
The vector field is injective over a half-edge $\e$, if its flux over that half-edge satisfies Equation~\eqref{eq:edge_inj}.
\label{prop:edge_inj}
\end{property}

We move on to vertex injectivity. We classify the injectivity on a vertex into three cases:

\begin{enumerate}
    \item If the point $p$ lies on a \textit{vertex} and is flowed somewhere, for the flow to be well-defined (and injective) there must be exactly one triangle with a vector that has negative flux to the vertex, and one triangle that has positive flux from the vertex. All other triangles have no flux with respect to the vertex. This is the case of a \textbf{regular injective} vertex (see Figure~\ref{fig:vertex_inj}, top left). The vector that has outgoing flow with respect to the vertex is chosen to define the flow. 

\item If the vertex is not flowed and stays in place, then all triangles in the 1-ring have zero vertex flux, implying the flow rotates around the vertex (Figure~\ref{fig:vertex_inj}, bottom left). This is a \textbf{vortex injective} vertex.

\item All other cases are \textbf{non-injective} vertices.
\end{enumerate}

\vspace{5pt}
\begin{property}[Vertex injectivity]
    \label{prop:vert_inj}
   The vector field is injective over a vertex if its flux on that vertex is either regular injective or vortex injective, per the classification above.
\end{property}

Our goal is thus to characterize a family of vector fields that satisfy the two above  properties for injectivity.
\begin{algorithm}[t]        
\SetKwInOut{Input}{Input}
    \SetKwInOut{Output}{Output}
    \caption{Computing an injective vector field.}
    \label{alg:compute}
    \Input{Sphere-topology mesh $\mesh$; choice of source/sink vertices $s,t$; parameterized Laplacian $\lap$. }
	\Output{Vector field $\vv$.}
    Solve Equation~\eqref{eq:harmonic} with respect to $\lap,s,t$ to obtain $\z$.

    Compute the per-triangle gradient field  $\nabla_\tri \z $.
    
    \Return the vector field $\vv$ defined via $\vv_\tri  = \textbf{R}_\tri^{90^\circ}\nabla_\tri \z$.
\end{algorithm}

\section{Injective Flows via Harmonic Functions}

Towards producing an injective vector field and proving its injectivity, this section progresses in three steps:
\begin{enumerate}
    \item (Section~\ref{ss:harmonic}) Compute a harmonic function $\z$.

\item (Section~\ref{ss:vector_field}) Define a vector field $\vv$ as the $90^\circ$-rotated gradient of the harmonic function, and show that its flux $\flux$ is exactly equal to a 1-form $\delta$ defined by $\z$.

\item (Section~\ref{ss:injectivity_proof}) Use the 1-form's properties to show it satisfies the necessary injectivity properties from Section~\ref{ss:injective_flux}, thereby showing that $\flux$ satisfies them as well. 
\end{enumerate}
\subsection{Computing a Discrete Harmonic 1-Form.}
\label{ss:harmonic}

We define a discrete Laplace equation on the mesh: we choose two vertices to serve as source/target, designated $s,t$. We also assume to have a Laplacian $\lap$, with any desired positive weights $w_{ij}$ on each edge (e.g., these can either be positively-clamped cotan weights~\cite{pinkall1993}, or ones that are optimized). We  solve the discrete Laplace equation with boundary conditions defined by $s,t$ and using the Laplacian $\lap$:
\begin{equation}
\label{eq:harmonic}
\begin{array}{ll}
\z_i = 1 & i =s   \\
\z_i = -1 & i =t   \\
\parr{\lap \z}_i = 0  & \text{else}
\end{array}
\end{equation}
Similar to~\cite{gortler2006}, for any half-edge $\e=(i,j)$, we define  the 1-form 
\begin{equation}
\label{eq:1form}
    \delta_\e = {\z_j - \z_i},
\end{equation} 
assigning a real number to that half-edge. See Figure~\ref{fig:edge_inj}, right, for an example of such $\z$ and the resulting 1-form $\delta$. Note we assume $\z$ is such that the 1-form is non-vanishing\setlength{\columnsep}{5pt}
 on any edge ($\delta_\e\neq0$); otherwise, similarly to the  argument made in Appendix A in~\cite{gortler2006}, we can perturb the 1-form by an infinitesimal amount to make it non-vanishing, without changing any of the assumptions necessary for the proof.

\subsection{Constructing an Injective Vector Field from the Harmonic Function}
\label{ss:vector_field}
We define our vector field on each triangle $\tri$ as the gradient vector rotated by $90^\circ$:
\begin{equation}
\label{eq:our_vector_field}
    \vv_\tri = \textbf{R}^{90^\circ}\nabla_\tri \z,
\end{equation}
which leads to the main theoretical result of this paper:
\vspace{5pt}
\begin{theorem}
\label{thm:injectivity}
    The flow induced by the vector field $\vv$, defined in Equation~\eqref{eq:our_vector_field}, is injective.
\end{theorem}
The following Section~\ref{ss:injectivity_proof} proves the above theorem. Similarly to Tutte's embedding~\cite{tutte1963}, injectivity holds for a very specific setting: a sphere-topology mesh, with one source and one sink -- see Figure~\ref{fig:non-injective-isolines} for a non-injective example when attempting the same technique on a mesh of a different genus.

\begin{wrapfigure}[11]{r}{0.4\columnwidth}
   \vspace{-14pt}
\includegraphics[width=\linewidth]{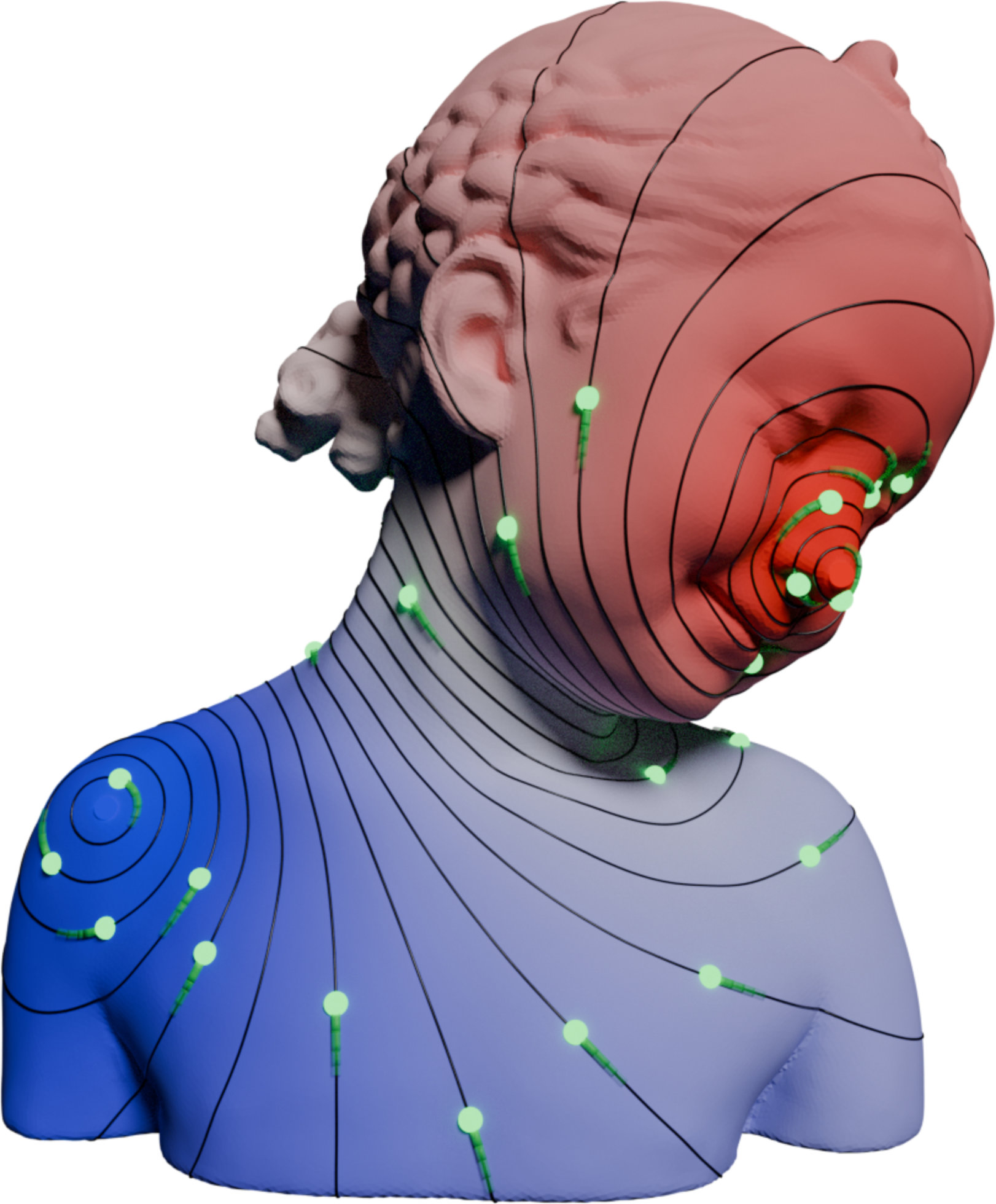}
\end{wrapfigure}
Geometrically, the rotated gradient vector field points in a direction tangent to a level set of $\z$; thus, the resulting injective flow defines a trajectory that marches along a fixed level set of $\z$, in a direction designated by the orientation of the mesh. The source and sink vertices $s,t$ serve as two vortex points (the flow rotates around them and they are stationary). Note, however, that when using Laplacians with random weights (Figure~\ref{fig:rand_lap},~\ref{fig:accuracy}) or optimized ones (Figures~\ref{fig:learn vf},~\ref{fig:optimize_sparse}), the level sets become more intricate and enable different flows.

    


    

    
         
        
        
    


\subsection{Injectivity of the Flow (Proof of Theorem~\ref{thm:injectivity})}
\label{ss:injectivity_proof}
We aim to provide a streamlined proof that highlights the key insights leading to the main result; as such, we prove all the necessary lemmas in Appendix~\ref{app:lemma}.  
To prove injectivity of the vector field, we need to prove that the injective properties  regarding its \emph{flux} $\flux$ (Section~\ref{ss:injective_flux}) hold. The key connection between the harmonic function $\z$ and the flux $\flux$ is that on each edge $\e$, the flux $\flux_\e$ is exactly equal to the sign of the 1-form $\delta_\e$.  

\vspace{5pt}
\begin{lemma}[flux/1-form equivalence]
\label{lem:flux_1form_equiv}
    The edge flux of the vector field $u$ defined in Equation~\eqref{eq:our_vector_field}  is equal on all edges to the 1-form defined in Equation~\eqref{eq:1form}:
    \begin{equation}
        \flux_\e = \sign{\delta_\e}.
    \end{equation}
\end{lemma}

 
We use this equivalence to prove  edge/vertex flux injectivity.

To prove \emph{edge injectivity}, we need to prove Property~\ref{prop:edge_inj} in Section~\ref{ss:injective_flux}: $\flux_{\e} = -\flux_{\twine}$. Per Lemma~\ref{lem:flux_1form_equiv}, this is equivalent to proving $\delta_{\e}= - 
\delta_{\twine}$. This last equation is immediate to deduce from 
Equation~\eqref{eq:1form} which implies the anti-symmetry of  $\delta$ (see also the half-edges shared between the two triangles in Figure~\ref{fig:edge_inj}, right).  Hence, we assume the vector field is edge-injective and turn to proving \emph{vertex injectivity}:



\vspace{5pt}
\begin{theorem}[Vertex injectivity]
\label{thm:vert_inj}
    Let $\mesh$, $\z$ be as above. Then the vector field defined by Equation~\eqref{eq:our_vector_field} is vertex-injective as defined in Property~\ref{prop:vert_inj}, Section~\ref{ss:injective_flux}.
\end{theorem}
To prove this theorem, we draw a connection between the flux conditions of Campen et al.~\cite{campen2016} and the behavior of indices of harmonic 1-forms described in~ Gortler et al.~\cite{gortler2006}.

\begin{proof}[Proof of Theorem~\ref{thm:vert_inj}]
We turn to~\cite{gortler2006}. We consider their Theorem 3.5 (essentially a discrete Poincaré-Hopf theorem), and interpret it in our scenario. They define an \textit{index} $\text{ind}(v_o)$ assigning an  integer for each vertex $v_o$. The index of the vertex $v_o$ is defined by stepping over its 1-ring edges in order, each oriented in outgoing direction $(o,1),(o,2)...,(o,k)$, and checking how many times the 1-form $\delta_{o,j}$ changes signs, with the index, $\ind{v_o}$, being equal to $1$, $0$ or $-1$ based on zero, two, or more than two sign changes, respectively (see Figure~\ref{fig:vertex_inj}, right column).

The proof hinges on the following Lemma~\ref{lem:us_gortler_equiv}, which shows  an equivalence between the aforementioned index classification and our version (Section~\ref{ss:injective_flux},  Property~\ref{prop:vert_inj}) of the classification from Campen et al.~\cite{campen2016} of  injective and non-injective vertices. Figure~\ref{fig:vertex_inj} portrays this equivalence between the left and right images in each row.
\begin{lemma}
\label{lem:us_gortler_equiv}
The vertex $v$ has index $1$ iff it is a \textit{regular} injective vertex, an index $0$ iff it is a \textit{vortex} injective vertex, and index $-1$ iff it is a \textit{non-injective} vertex, per the definition in Section~\ref{ss:injective_flux}, Property~\ref{prop:vert_inj}. 
\end{lemma}

Hence, to ensure all vertices are injective, all we need to show is that there are no vertices with $\text{ind}(v) = -1$. We prove a slightly stronger version of this claim using standard arguments~\cite{gortler2006}:

\begin{lemma}
\label{lem:non_neg_index}
    Let $\mesh$, $\z$ and $\delta$ be as described above. Then all vertices have index $0$, except for the source and sink vertices $s,t$, which have index 1.
\end{lemma}
\vspace{-8pt}
Thus, no vertex has negative index, entailing that all vertices are injective, and we have proven Theorem~\ref{thm:vert_inj}.
\end{proof}

With this, we have proven that both edge injectivity and vertex injectivity hold, thereby proving Theorem~\ref{thm:injectivity} and global injectivity of the flow.

\subsection{Vector Field and Flow computation}
We summarize the vector field construction in Algorithm~\ref{alg:compute}. Each step of this algorithm is differentiable, and thus it provides a differentiable parameterization of injective vector fields, as a function of the Laplacian's (arbitrary) edge weights $\lap_{ij}$.

\begin{figure}[t]
     \centering
     \begin{subfigure}[b]{0.49\columnwidth}
         \includegraphics[width=\textwidth]{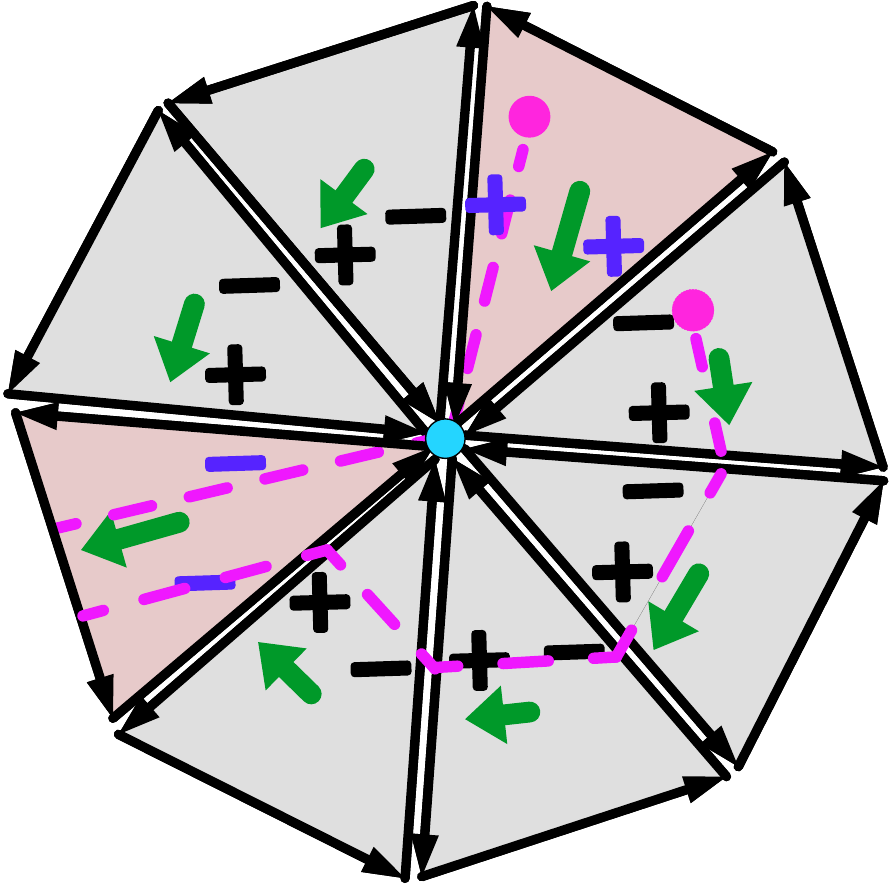}
         \captionsetup{justification=centering}
         \caption*{Injective (regular) vertex flow\\\textbf{Two} even-parity triangles}
     \end{subfigure}
     \hfill 
     \begin{subfigure}[b]{0.49\columnwidth}
         \includegraphics[width=\textwidth]{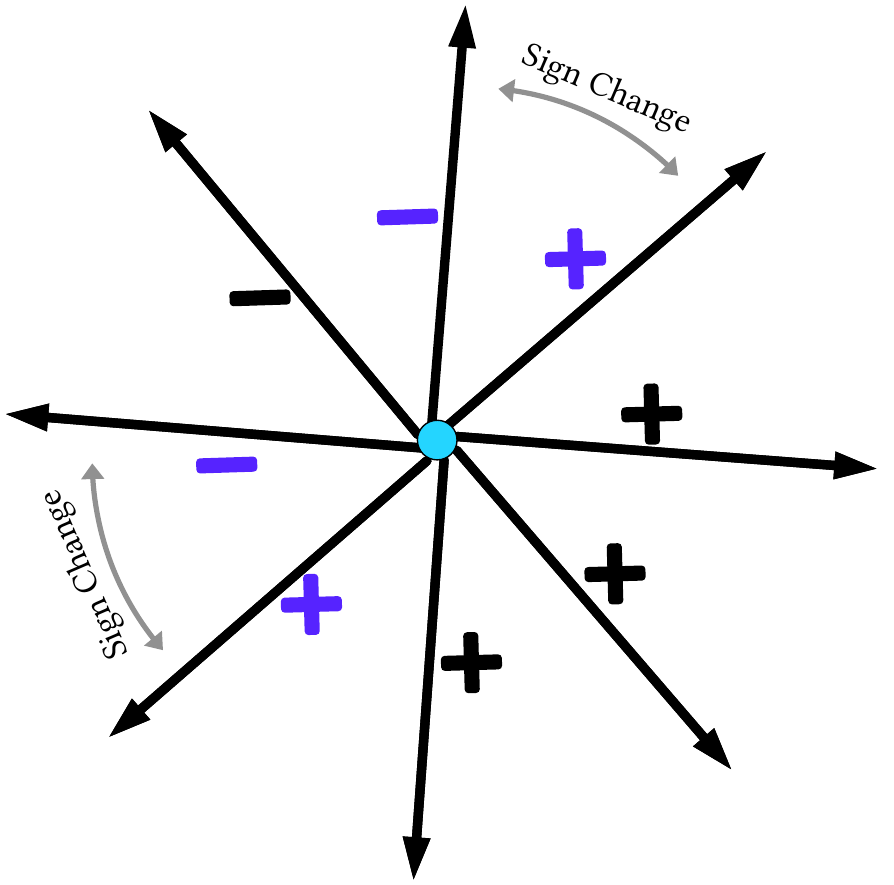}
         \captionsetup{justification=centering}
         \caption*{ Index = 0  per~\cite{gortler2006}\\\textbf{Two} sign changes}
     \end{subfigure}

     \begin{subfigure}[b]{0.49\columnwidth}
         \includegraphics[width=\textwidth]{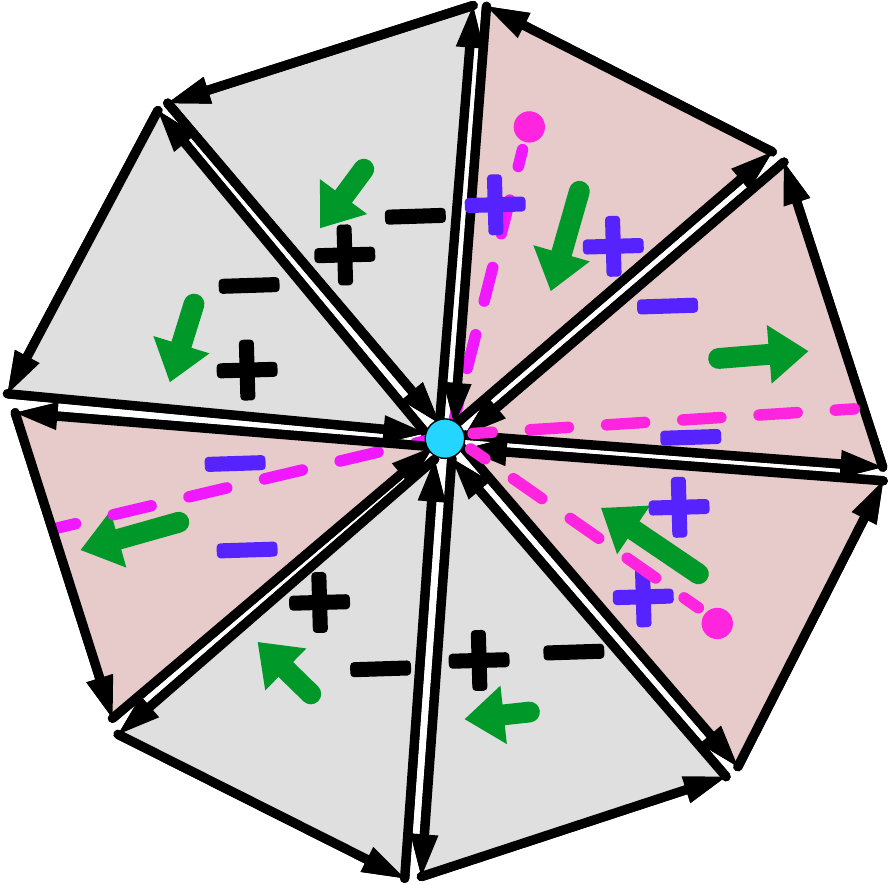}
         \captionsetup{justification=centering}
         \caption*{Non-injective vertex flow\\ \textbf{More than two} even-parity triangles}
     \end{subfigure}
     \hfill 
     \begin{subfigure}[b]{0.49\columnwidth}
         \includegraphics[width=\textwidth]{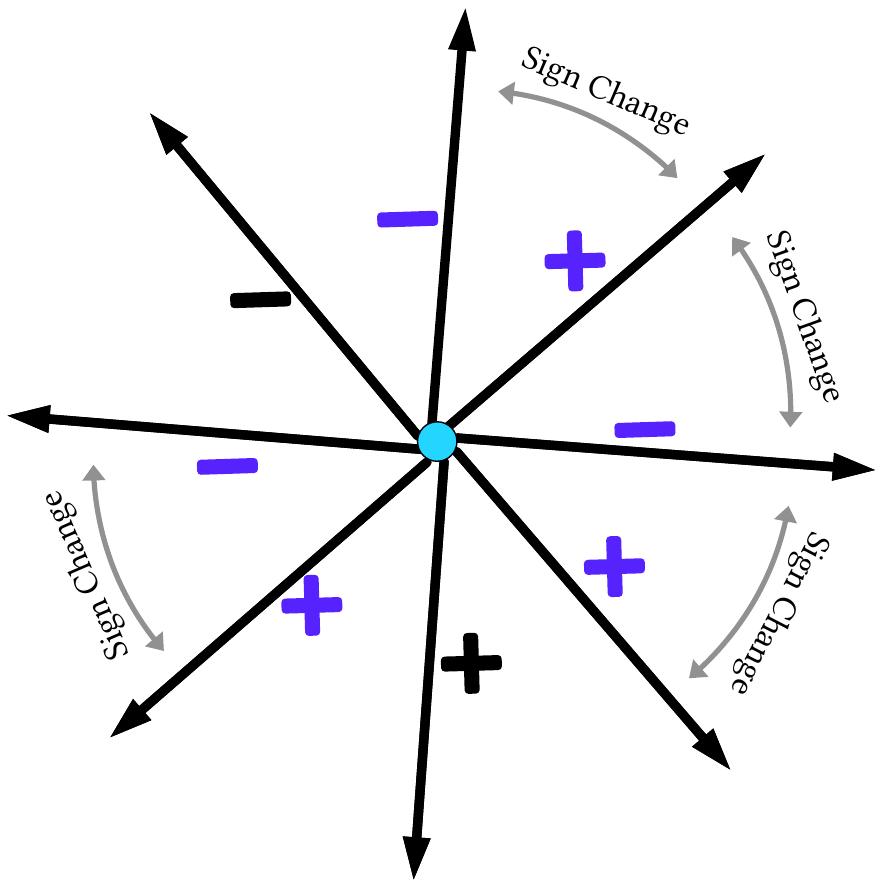}
         \captionsetup{justification=centering}
         \caption*{  Index = -1  per~\cite{gortler2006}\\\textbf{More than two} sign changes}
     \end{subfigure}
     \begin{subfigure}[b]{0.49\columnwidth}
         \includegraphics[width=\textwidth]{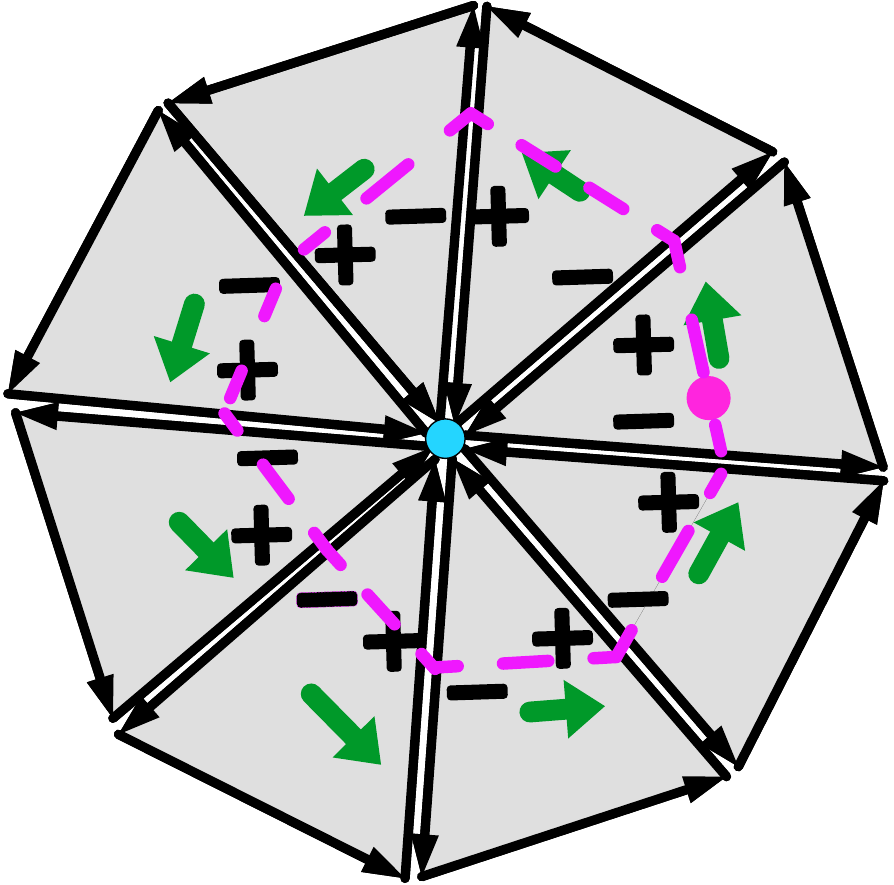}
         \captionsetup{justification=centering}
         \caption*{Injective (vortex) vertex flux\\\textbf{Zero} even-parity triangles}
     \end{subfigure}
     \hfill 
     \begin{subfigure}[b]{0.49\columnwidth}
         \includegraphics[width=\textwidth]{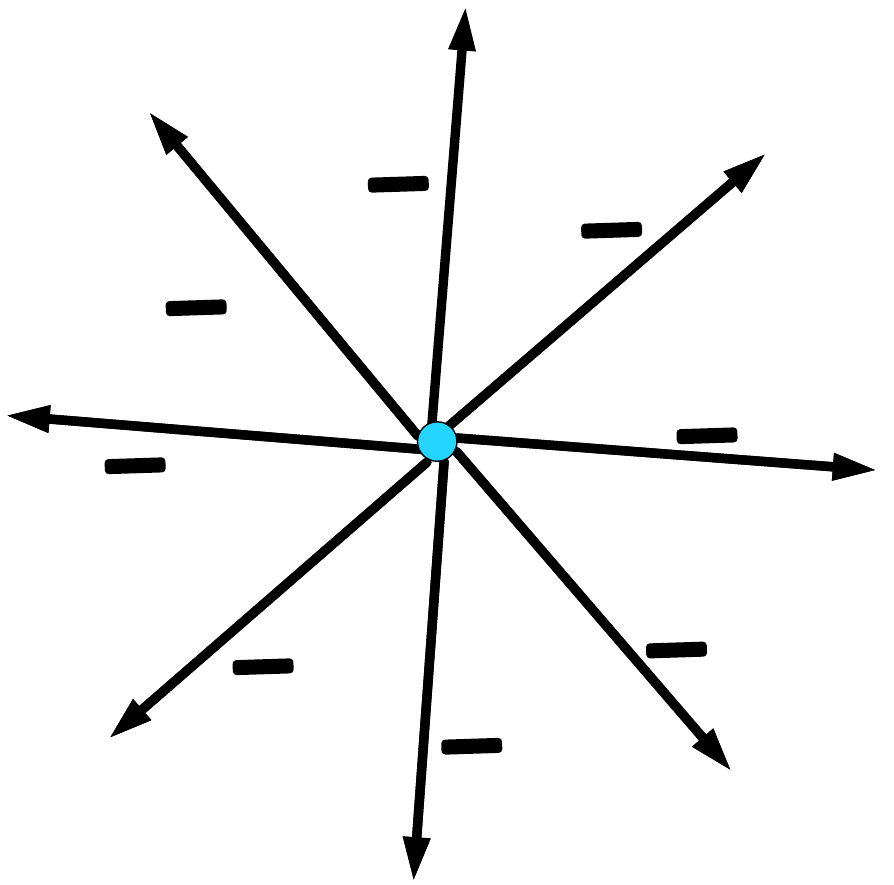}
         \captionsetup{justification=centering}
         \caption*{  Index = 1  per~\cite{gortler2006}\\\textbf{Zero} sign changes}
     \end{subfigure}
     
     \caption{\textbf{Vertex flux injectivity vs. vertex index classification}. Left column: the three types of vertex fluxes  on the center vertex of a 1-ring, as characterized in~\cite{campen2016}. Triangles with even parity (same sign on both edges) highlighted in red. Right column: the corresponding diagrams of the 1-forms from~\cite{gortler2006}.  Note that in each row, the sign changes on the right happen exactly on the triangles with even parity on the left. }
     \label{fig:vertex_inj}
\end{figure}
Lastly, to flow a given point $p\in\mesh$ on a  triangle $\tri$ with respect to the vector field $\vv$ for a given amount of time $T\in\real$, we propose a simple, efficient, and accurate algorithm, described in Algorithm~\ref{alg:flow}, which walks on the level sets of $\z$. The flow is piecewise linear and only needs to be evaluated at each intersection point with an edge. The choice of which edge is the next to be intersected, as well as the exact location on the edge on which the intersection happens, can all be easily deduced by simple arithmetic involving the value $c$ of the level set and the values of $\z$ on the vertices. 
Since the algorithm relies on the initial level set to find the intersection points, it does not suffer from accumulated error and can self-correct to always follow the level set (see Table~\ref{tab:to_fro_error}).
\begin{figure*}[t]
     \centering
         \includegraphics[width=\textwidth]{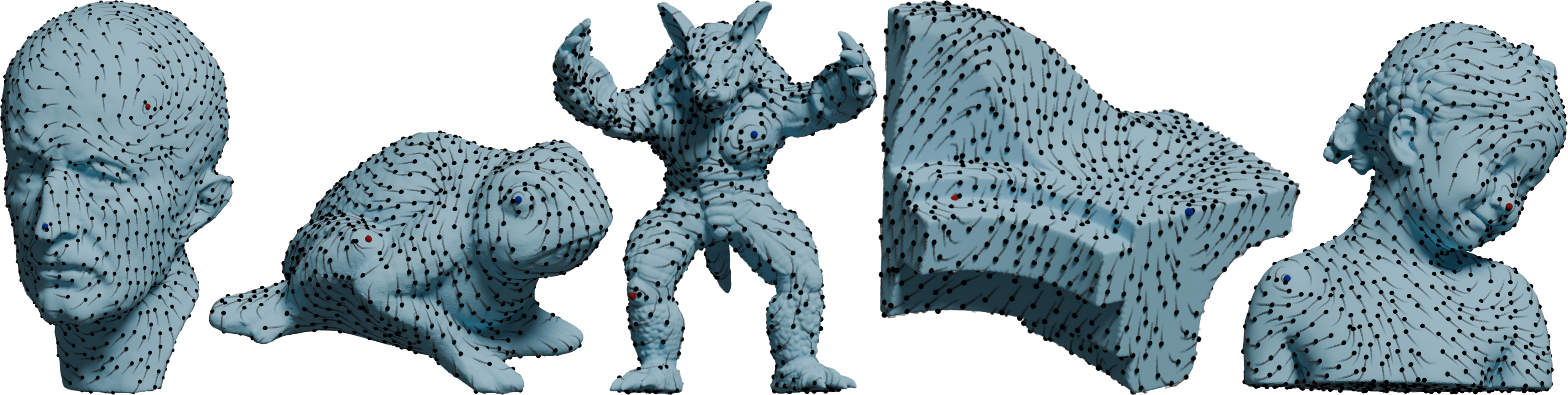}
    \caption{\textbf{Injective flows generated via Laplacians with arbitrary edge weights.} We sample random positive weights on the Laplacian's edges, which leads to more irregular flows. Theorem~\ref{thm:injectivity} guarantees the flows are injective.}
    \label{fig:rand_lap}
\end{figure*}
\begin{algorithm}[h]
    \SetKwInOut{Input}{Input}
    \SetKwInOut{Output}{Output}
    \caption{Flowing a point along the level set.}
    \label{alg:flow}
    
    \Input{Mesh $\mesh$, harmonic function $\z$, starting point $p$, vector field $\vv$.}
    \Output{End point $p_{\text{end}}$.}

    Store $c\in\real$, as the level set of the harmonic function $\z$ evaluated at  point $p$: if the barycentric coordinates of $p=\alpha v_i+\beta v_j+\gamma v_k$, then $c = \alpha \z_i + \beta \z_j + \gamma \z_k$

    Set the currently visited triangle, $\tri$, to the one that contains $p$ 

    Set the last intersection edge $\e$ of $\tri$ to be the one with which the point intersects when flowed in reverse.
    
    Set the remaining time to $T=1$

    \While{$T > 0$}{
        Out of the two edges that are not $\e$, find the edge $(i,j)$ of $\tri$ with which the isoline intersects, by choosing the one satisfying $c \leq\max\parr{ \z_i,\z_j} \text{ and } c\geq\min\parr{\z_i,\z_j}$

        Update the intersection edge, $\e=(i,j)$ 
    
        Find the intersection point between the isoline and the edge, solving $s\z_i + (1-s)  \z_j = c$ for $s$. 
         
        Set the intersection point with the next triangle as $p_{i+1} = sv_i+(1-s)v_j$, and the current triangle $\tri$ as the triangle adjacent to the half-edge $(j,i)$.
        
        Deduct the time it took to cross the triangle from $T$: $T=T-\norm{p_{i+1}-p_i}/\norm{\vv_\tri}$.
        
        \If{$T\leq0$}{
            \Return $p_{\text{end}}=p_{i+1}+T\vv_\tri$\;
        }
    
    }

\end{algorithm}

\section{Results}
As our main contribution is the theoretical Theorem~\ref{thm:injectivity}, we next empirically verify its correctness, show qualitative examples of injective flows, and also validate the ability of the differentiable pipeline to enable gradient-based optimization. We note that the injectivity of a vector field does not depend on the magnitude of the vectors and hence opt to normalize all vectors to unit norm.

\subsection{Empirical Evidence of Injectivity and Accuracy}

\begin{figure}[t]
     \centering
         \includegraphics[width=\columnwidth]{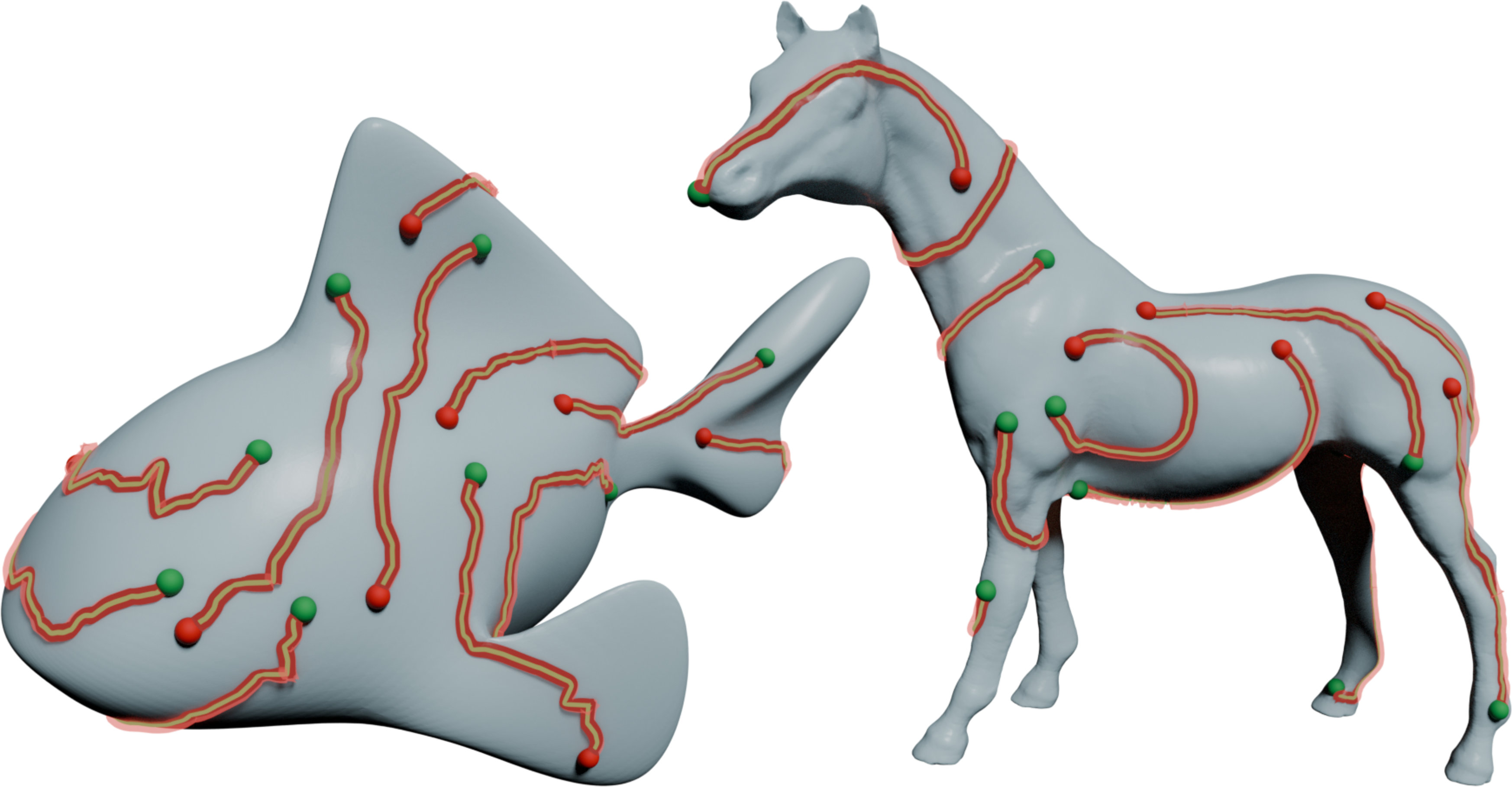}
         \caption{\textbf{Qualitative validation of the injectivity and accuracy of the flow.} We generate flows from random-weight Laplacians, and  advect sampled points (in red) forward (thick red curves) in time, then for each take its terminal location (green)  and flow it backwards in time (thin yellow curve), to verify it yields the exact same trajectory and reaches back to the starting point.}
         \label{fig:accuracy}
     \end{figure}
     \begin{figure}
     \hfill 
         \centering
         \includegraphics[width=\columnwidth]{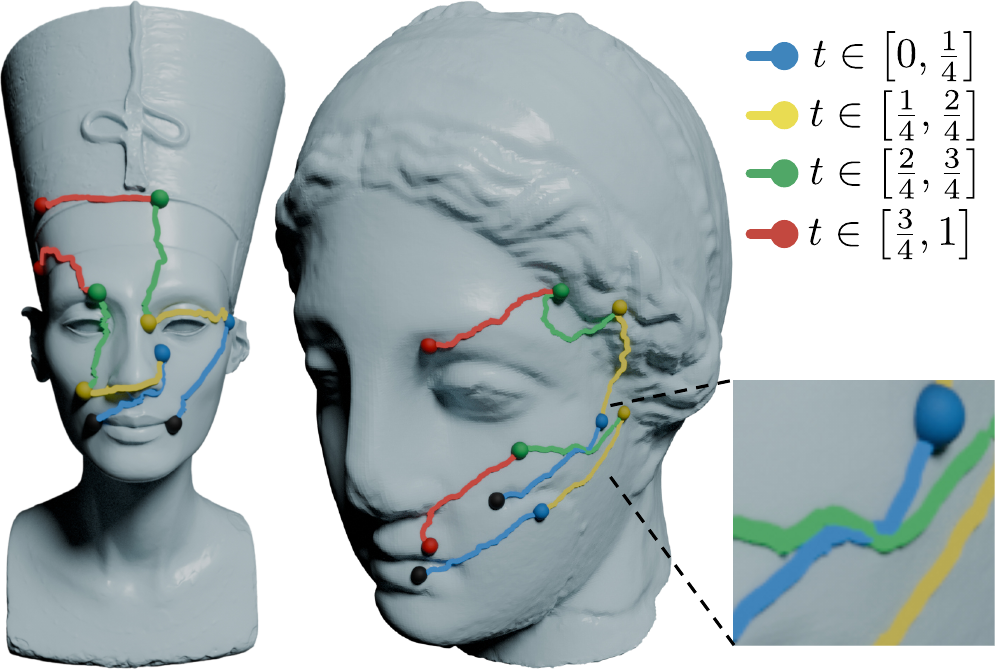}
         \caption{\textbf{Discrete time-dependent flows.} By alternating between four harmonic functions based on time, the flow produced can become extremely elaborate, each harmonic function's path  visualized in different color. Intersections can happen between different moments in time as shown in the zoom-in; this does not violate the injectivity of the flows.}
         \label{fig:multi}
    
\end{figure}

We begin by  empirically validating the injectivity of the flows, as well as the accuracy of our advection algorithm, Algorithm~\ref{alg:flow}. We flow a given point for a certain amount of time, and  then flip the vector field and flow the resulting point ``back in time'' -- we visualize this experiment in Figure~\ref{fig:accuracy} with red/yellow for forward/backward paths, validating injectivity as well as exhibiting the accuracy of the algorithm. For quantitative evaluation,  we repeat this process for random source and sink vertices alongside a random initialization of Laplacian weights, sample  $100,000$ random points, and measure the error between the start point and the point flowed back via the two steps above.  We  report the aggregated Mean Absolute Error (MAE) and Maximum Absolute Error (MXE) in Table~\ref{tab:to_fro_error}. No error exceeds $10^{-8}$ on average and $10^{-5}$ at maximum. A few models with triangles that have very small area or are near-degenerate result in less accurate results, mainly because of inaccuracy in numerics of the harmonic equation/gradient operator leading to a slight drift in the calculation of the time it takes to cross an edge.

\subsection{Flows Defined from Different Types of Laplacians}

Theorem~\ref{thm:injectivity}  guarantees an injective flow for any valid configuration of strictly positive Laplacian weights.  In Figure~\ref{fig:cot_and_mvc} we show flows computed via the cotangent Laplacian~\cite{pinkall1993} (clamped to have only positive weights) and via the Mean Value Coordinates (MVC)~\cite{floater2003mean} Laplacian. These lead to smooth flows along the harmonic isolines.
To show more elaborate flows, Figure~\ref{fig:rand_lap} shows results for Laplacians with arbitrary entries, which result in more oscillating vector fields and trajectories (see also Figure~\ref{fig:accuracy}). In all cases, the flows are injective. 

\subsection{Discrete Time-Dependent Flows}
\label{ss:multi}
In the continuous case, flows are often defined by time-dependent vector fields $\vv(p,t)$~\cite{chen2018neural}. Following our aim to provide a discrete counterpart of continuous flows, we can produce a \textit{time-discrete} analogue of the above, where we switch between different harmonic functions discretely, at different time intervals -- Figure~\ref{fig:multi} shows trajectories of points where the harmonic function governing the flow is changed four times. The resulting flow is more elaborate and, as shown on the right, can yield trajectories that intersect, although the flow remains injective (the two points each traverse the intersection point at different moments in time).

\subsection{Gradient-Based Injective Field Optimization} \label{sec:learn vf}

All steps of Algorithm~\ref{alg:compute} are differentiable, specifically with respect to the entries of the Laplacian; hence, by leveraging gradient-based optimization, we can effectively optimize the Laplacian edge weights~\cite{sun2024tuttenet,aigerman2024generative} thereby optimizing the vector field while maintaining its injectivity. We next show preliminary experiments validating the optimizeability of this representation. We implement the algorithm in PyTorch, which provides automatic differentiation. We use a simple optimization scheme, via the Adam~\cite{kingma2015} optimizer with a fixed step size of $0.1$, run for $5000$ iterations.

\paragraph*{Approximating  non-injective vector fields.}


In Figure~\ref{fig:learn vf}, we show a result of optimizing a randomly-initialized injective vector field to align with a given, non-injective vector field (obtained via~\cite{crane2010}), by minimizing the mean squared error (MSE) between the per-face target vectors and the optimized injective field. This leads to an injective vector field that still closely mimics the broader directional trends of the target field, as shown in the blowups.

\paragraph*{Satisfying sparse vector constraints.}
\begin{figure}[t]
     \centering
     
     \hfill 
     \begin{subfigure}[b]{0.32\columnwidth}
         \includegraphics[width=\textwidth]{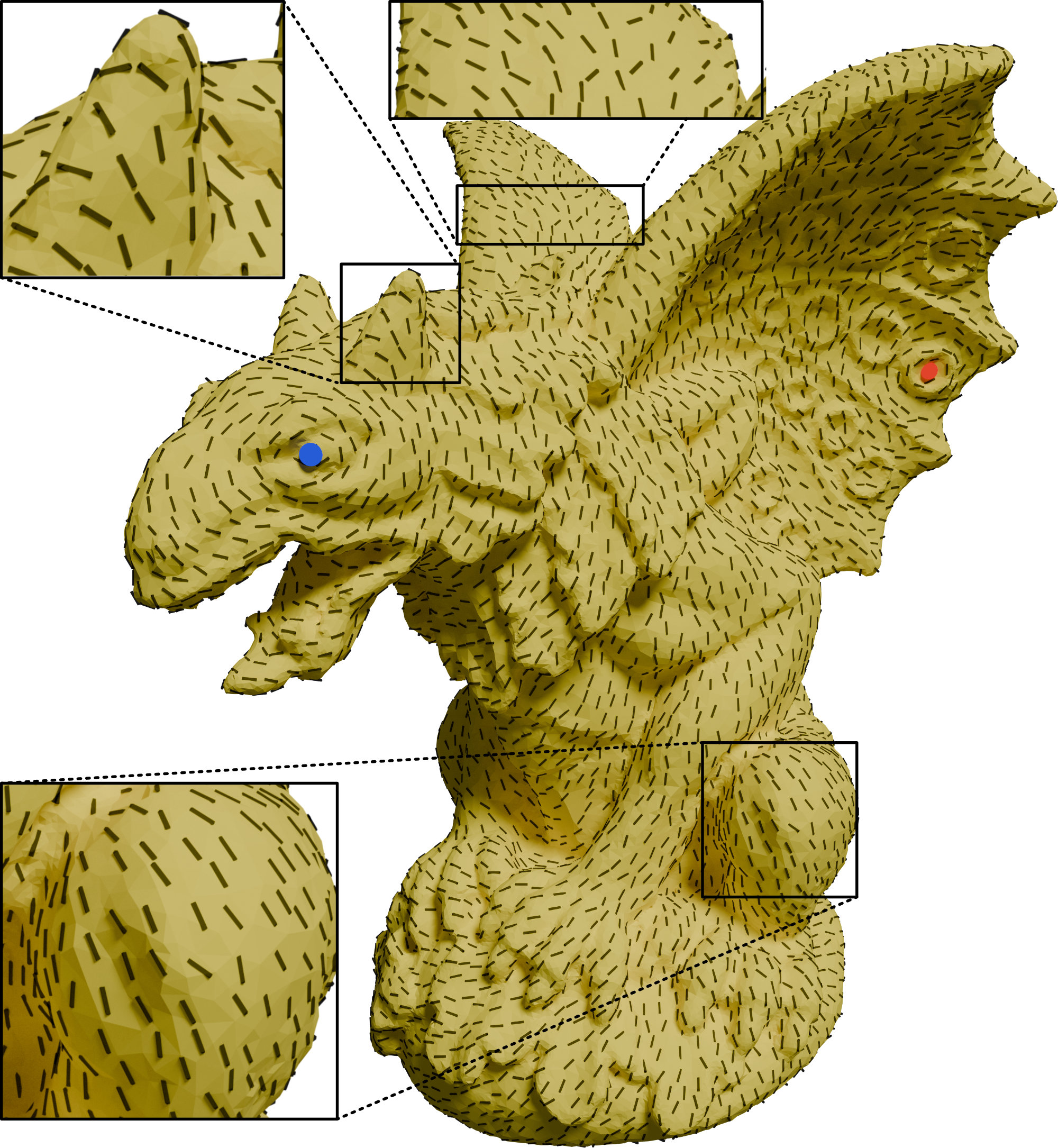}
         \caption{Initial injective vector field}
     \end{subfigure}
     \hfill
     \begin{subfigure}[b]{0.32\columnwidth}
         \includegraphics[width=\textwidth]{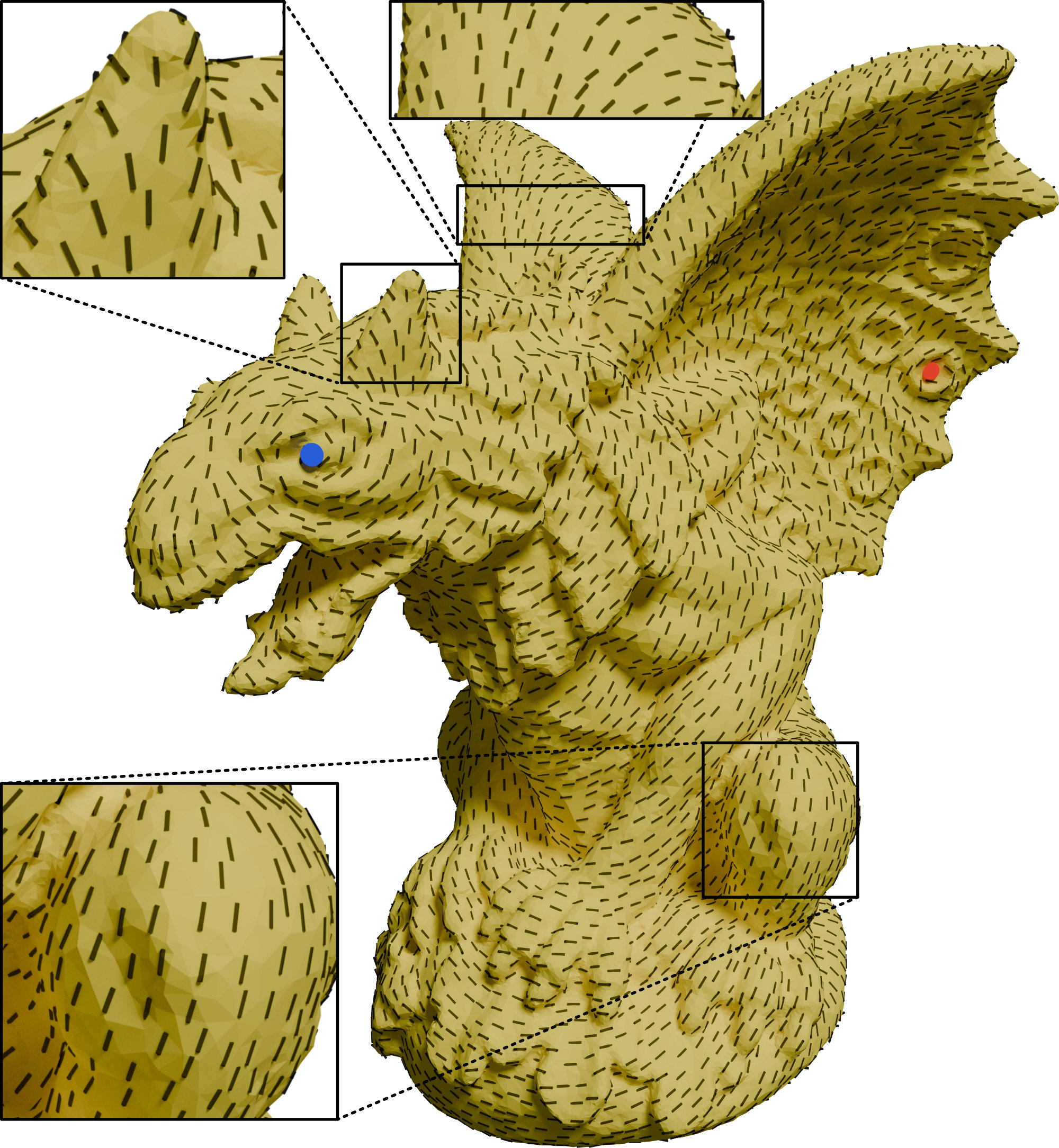}
         \caption{Optimized injective vector field}
     \end{subfigure}
     \begin{subfigure}[b]{0.32\columnwidth}
         \includegraphics[width=\textwidth]{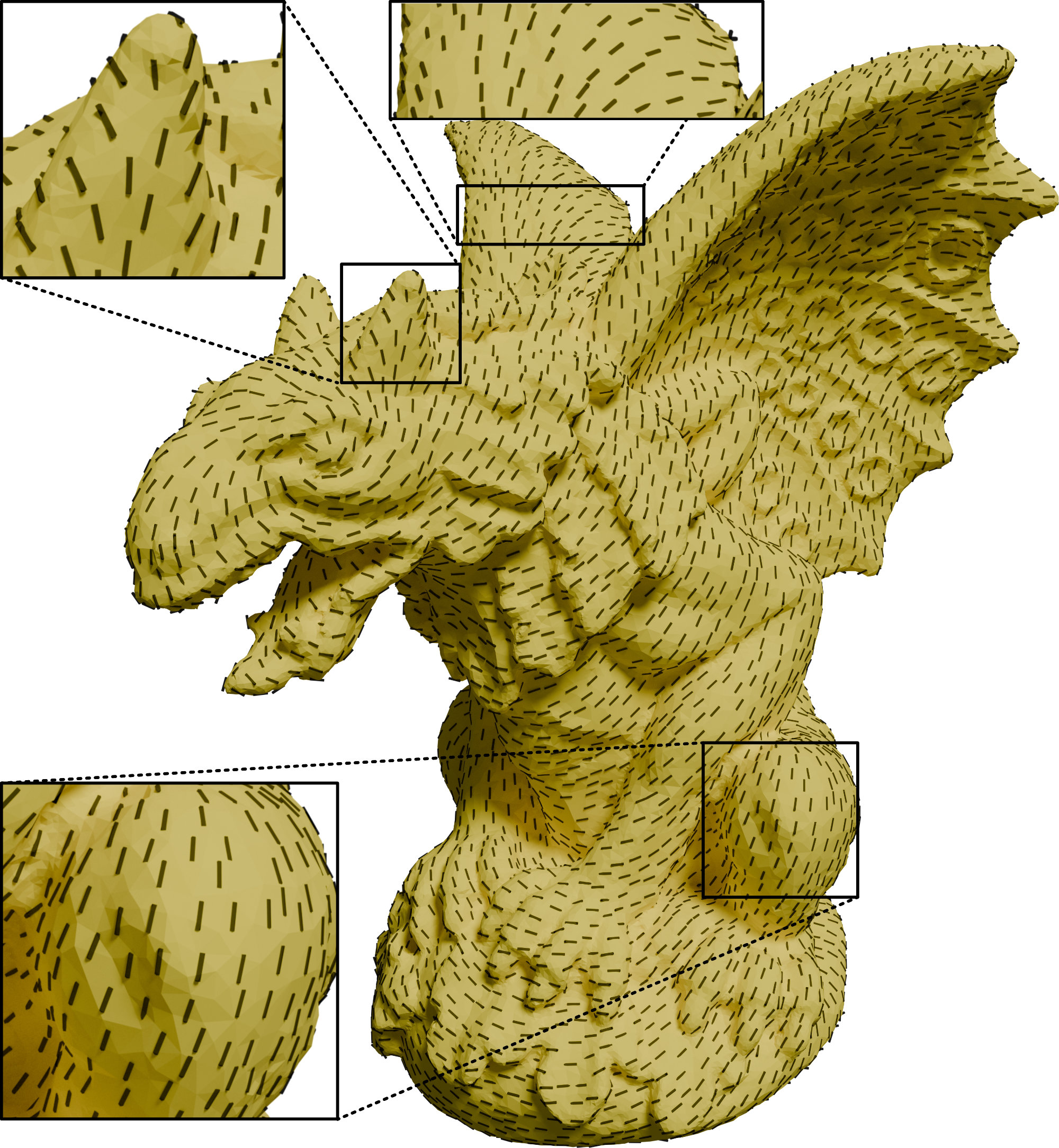}
         \caption{Target non-injective vector field}
     \end{subfigure}
    \caption{\textbf{Optimizing an injective vector field to approximate a  non-injective target vector field}.  The optimized field aligns to the target in many places, without sacrificing injectivity.}
    \label{fig:learn vf}
\end{figure}
\begin{figure}[t]
    \centering
     \begin{subfigure}[b]{0.495\columnwidth}
         \includegraphics[width=\textwidth]{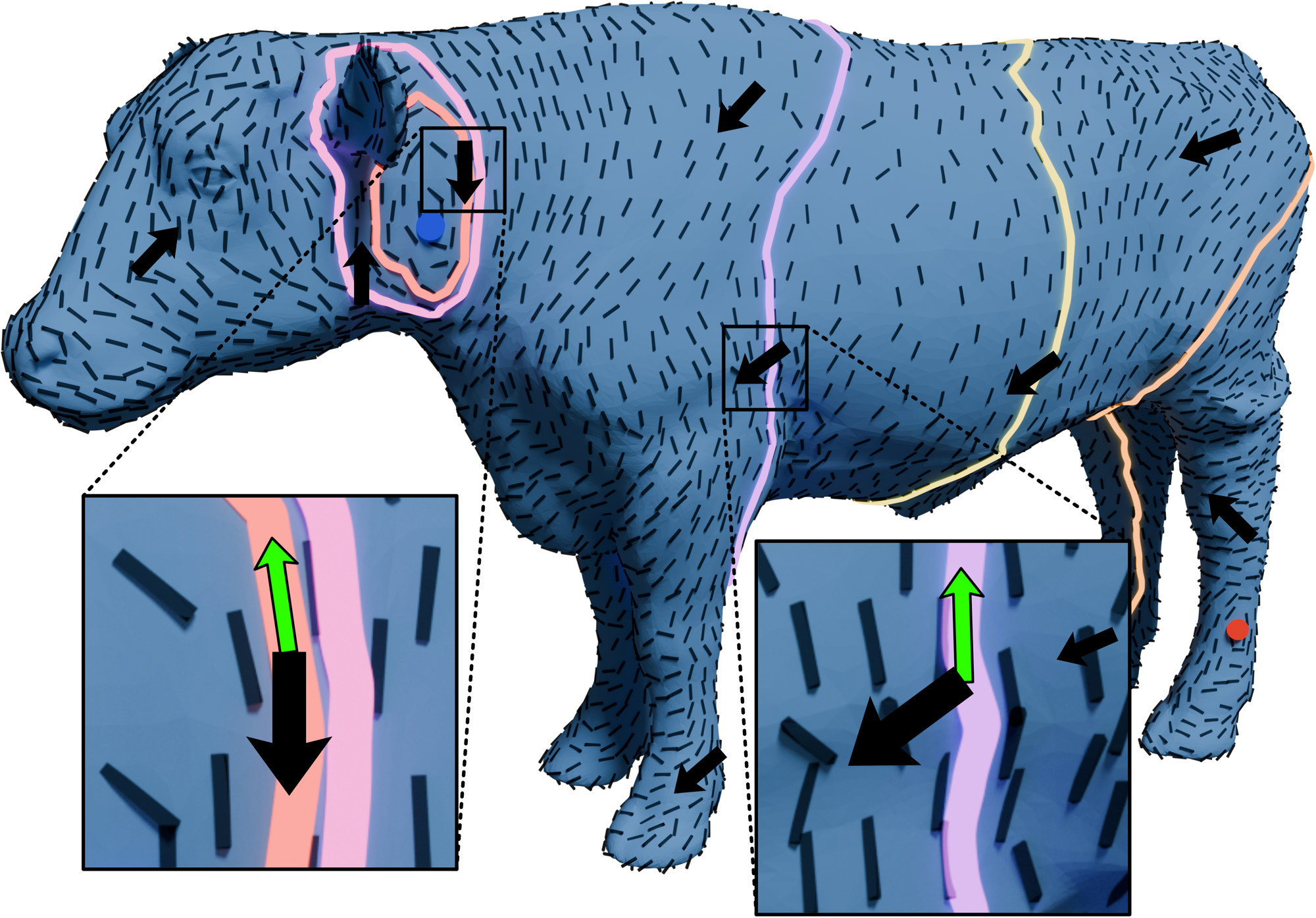}
         \caption{Initial injective field}
         \label{subfig: vf constraint before}
     \end{subfigure}
     \begin{subfigure}[b]{0.494\columnwidth}
         \includegraphics[width=\textwidth]{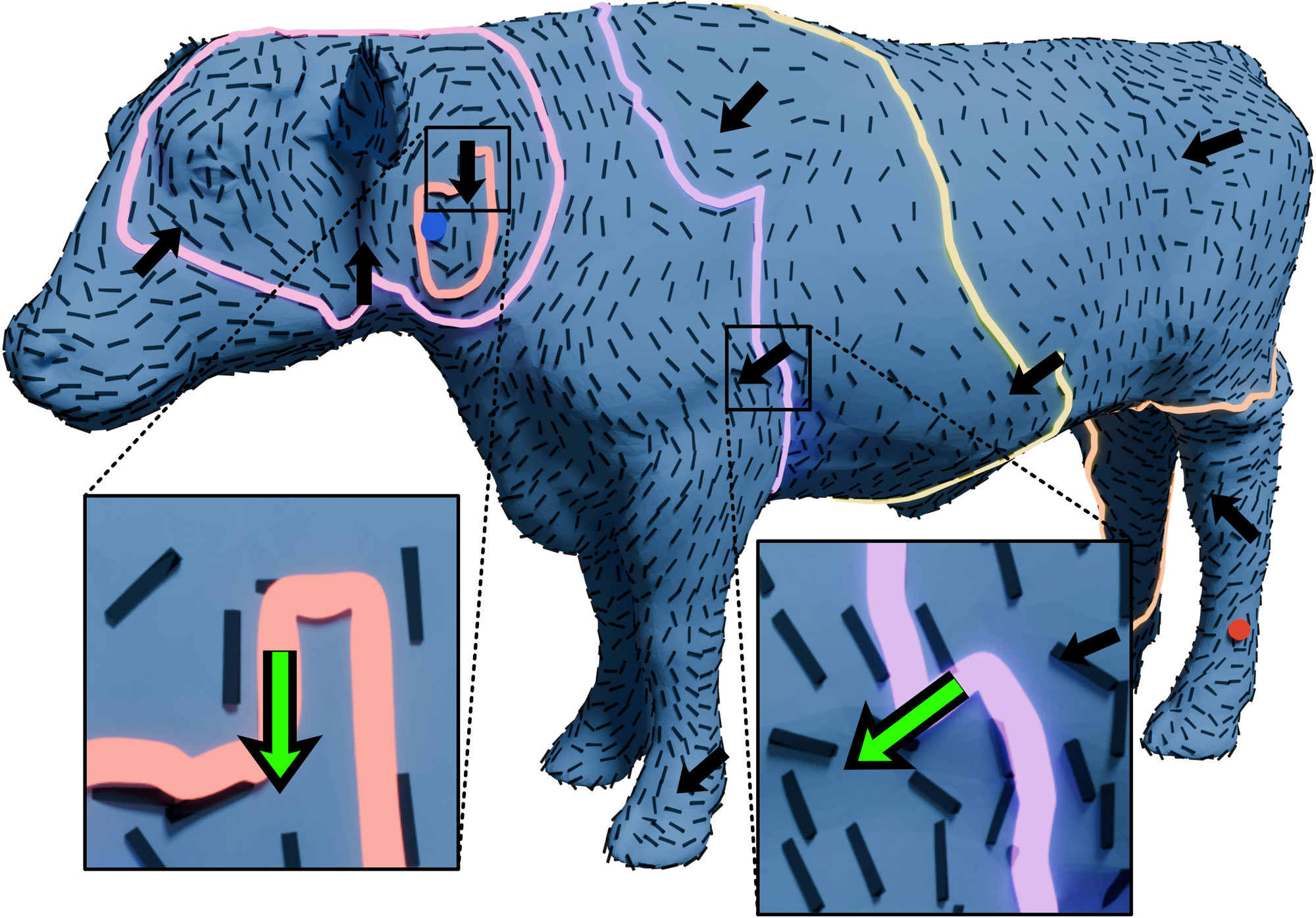}
         \caption{Optimized injective field}
         \label{subfig: vf constraint after}
     \end{subfigure}
    \caption{\textbf{Optimizing an injective vector field to satisfy sparse constraints}. The vector field is optimized to fit to directional constraints (black arrows). In the insets, the injective vector field's corresponding direction is visualized as a red arrow. The continuous trajectories along five of the constrained faces are rendered as colored curves.}
    \label{fig:optimize_sparse}
\end{figure}

We evaluate the ability of our optimization to satisfy sparse directional constraints across specific faces of a mesh. In Figure \ref{fig:optimize_sparse}, black arrows denote ten user-defined constraint vectors. We minimize the MSE loss between the injective vector field and these vectors, exclusively on the constrained faces. We visualize the continuous trajectories passing through five of these constrained faces on the mesh. As shown in the insets of Figure \ref{subfig: vf constraint before}, the initial vector field (red arrows) naturally deviates from the arbitrary constraints imposed by the target vectors, and is corrected by optimization to achieve near-perfect alignment, sometimes via rather contorted paths.

\section{Conclusion}

Theorem~\ref{thm:injectivity} draws a connection between harmonic  functions and injective flows and proposes a Tutte-like method for computing them. We hope that in the coming years, this method will find applications within the geometry processing community similar to those of Tutte's embedding, e.g., producing injective initializations for vector field optimization algorithms, serving as a fast method to construct an injective vector field, or optimizing injective fields. 

While our main contribution is theoretical, on a practical level, our method comes with several limitations, which we hope to solve in future work. First, the method is limited to specific topology and boundary conditions. We believe that extensions to some other specific boundary conditions on specific domains are possible. Furthermore, similarly to Tutte's embedding, injective extensions to non-euclidean domains~\cite{aigerman2016hyperbolic} can lead to a more general and less constrained construction. We also note that we have only shown a few ``proof-of-concept'' experiments to validate the capability of optimizing the vector field to specific goals, as we focused mainly on the injectivity of the flow, and used PyTorch's vanilla autodiff with Adam~\cite{kingma2015} that does not use adaptive step size nor line search. The current optimization framework is  insufficient on its own to optimize for flows that can be used in downstream tasks. As our top future priority, we mark the goal of devising a framework that can output production-ready flows. This requires regularizing the vector fields, e.g., via as-Killing-as-possible~\cite{benchen2010killing} or a vector Dirichlet energy ~\cite{stein2020simple}.  Lastly, our parameterization of injective flows via gradients of a family of harmonic functions does not cover the full space of injective vector fields. As several of our flows are composed (as in Figure~\ref{fig:multi}), the space of injective flows is increased, and we set as important future work to show (both theoretically and practically) that with enough such compositions, any flow can be attained.

\begin{table}[t]
  \centering
  \small
  \begin{tabular*}{\linewidth}{@{\extracolsep{\fill}} l r c c @{}}
    \toprule
    \textbf{Mesh} & \textbf{Faces} & MAE & MXE \\
    & & \scriptsize($10^{-9}$m) & \scriptsize($10^{-6}$m) \\
    \midrule
    Spot & 5,856 & $<$0.001 & $<$0.001 \\
    Cow & 8,626 & $<$0.001 & $<$0.001 \\
    Homer & 12,000 & $<$0.001 & $<$0.001 \\
    Fandisk & 12,946 & $<$0.001 & $<$0.001 \\
    Cheburashka & 13,334 & $<$0.001 & $<$0.001 \\
    Blub & 14,208 & $<$0.001 & $<$0.001 \\
    Horse & 39,698 & 0.173 & 0.075 \\
    Stanford Bunny & 69,664 & $<$0.001 & $<$0.001 \\
    Max Planck & 98,260 & $<$0.001 & $<$0.001 \\
    Nefertiti & 99,938 & $<$0.001 & $<$0.001 \\
    Vase-Lion & 127,104 & $<$0.001 & $<$0.001 \\
    Bimba & 149,524 & $<$0.001 & $<$0.001 \\
    Gargoyle & 171,112 & $<$0.001 & $<$0.001 \\
    Igea & 268,686 & $<$0.001 & $<$0.001 \\
    Armadillo & 345,944 & $<$0.001 & $<$0.001 \\
    Frog & 392,552 & $<$0.001 & $<$0.001 \\
    Lucy & 525,814 & 0.002 & 0.001 \\
    Crab & 602,316 & 1.011 & 3.658 \\
    Dragon & 721,890 & 0.071 & 0.762 \\
    Octopus & 1,083,284 & 0.863 & 1.230 \\
    \bottomrule
  \end{tabular*}
  \caption{\textbf{Quantitative evaluation of injectivity and accuracy of the flow}. We measure the average error (MAE) and max error (MXE)  between points and their final location when flowed back and forth along vector fields derived from random Laplacians. }
  \label{tab:to_fro_error}
\end{table}

\bibliographystyle{eg-alpha-doi} 
\bibliography{main-ref-sub}       

@inproceedings{
    chen2023riemannianfm,
    title={Riemannian Flow Matching on General Geodesics},
    author={Ricky T. Q. Chen and Yaron Lipman},
    year={2023},
}

@inproceedings{degoes2016,
author = {do Goes, Fernando and Desbrun, Mathieu and Tong, Yiying},
title = {Vector field processing on triangle meshes},
year = {2015},
isbn = {9781450339247},
publisher = {Association for Computing Machinery},
address = {New York, NY, USA},
url = {https://doi.org/10.1145/2818143.2818167},
doi = {10.1145/2818143.2818167},
booktitle = {SIGGRAPH Asia 2015 Courses},
articleno = {17},
numpages = {48},
location = {Kobe, Japan},
series = {SA '15}
}

@article{dong2005harmonic,
  title={Harmonic functions for quadrilateral remeshing of arbitrary manifolds},
  author={Dong, Shen and Kircher, Scott and Garland, Michael},
  journal={Computer aided geometric design},
  volume={22},
  number={5},
  pages={392--423},
  year={2005},
  publisher={Elsevier}
}

@inproceedings{stein2020simple,
  title={A simple discretization of the vector dirichlet energy},
  author={Stein, Oded and Wardetzky, Max and Jacobson, Alec and Grinspun, Eitan},
  booktitle={Computer Graphics Forum},
  volume={39},
  number={5},
  pages={81--92},
  year={2020},
  organization={Wiley Online Library}
}

@inproceedings{tong2006designing,
  title={Designing quadrangulations with discrete harmonic forms},
  author={Tong, Yiying and Alliez, Pierre and Cohen-Steiner, David and Desbrun, Mathieu},
  booktitle={Eurographics symposium on geometry processing},
  year={2006}
}

@article{vekhter2019weaving,
author = {Vekhter, Josh and Zhuo, Jiacheng and Fandino, Luisa F Gil and Huang, Qixing and Vouga, Etienne},
title = {Weaving geodesic foliations},
year = {2019},
issue_date = {August 2019},
publisher = {Association for Computing Machinery},
address = {New York, NY, USA},
volume = {38},
number = {4},
issn = {0730-0301},
url = {https://doi.org/10.1145/3306346.3323043},
doi = {10.1145/3306346.3323043},
journal = {ACM Trans. Graph.},
month = jul,
articleno = {34},
numpages = {22}
}

@inproceedings{lipman2023,
  author       = {Yaron Lipman and
                  Ricky T. Q. Chen and
                  Heli Ben{-}Hamu and
                  Maximilian Nickel and
                  Matthew Le},
  title        = {Flow Matching for Generative Modeling},
  booktitle    = {The Eleventh International Conference on Learning Representations,
                  {ICLR} 2023, Kigali, Rwanda, May 1-5, 2023},
  publisher    = {OpenReview.net},
  year         = {2023},
  url          = {https://openreview.net/forum?id=PqvMRDCJT9t},
  bibsource    = {dblp computer science bibliography, https://dblp.org}
}

@article{verhoeven2022dev2pq,
  title={Dev2PQ: Planar quadrilateral strip remeshing of developable surfaces},
  author={Verhoeven, Floor and Vaxman, Amir and Hoffmann, Tim and Sorkine-Hornung, Olga},
  journal={ACM Transactions on Graphics (TOG)},
  volume={41},
  number={3},
  pages={1--18},
  year={2022},
  publisher={ACM New York, NY}
}

@article{sharp2019vector,
  title={The vector heat method},
  author={Sharp, Nicholas and Soliman, Yousuf and Crane, Keenan},
  journal={ACM Transactions on Graphics (TOG)},
  volume={38},
  number={3},
  pages={1--19},
  year={2019},
  publisher={ACM New York, NY, USA}
}

@article{palmer2024lifting,
  title={Lifting directional fields to minimal sections},
  author={Palmer, David and Chern, Albert and Solomon, Justin},
  journal={ACM Transactions on Graphics (TOG)},
  volume={43},
  number={4},
  pages={1--20},
  year={2024},
  publisher={ACM New York, NY, USA}
}

@inproceedings{vaxman2016,
author = {Vaxman, Amir and Campen, Marcel and Diamanti, Olga and Bommes, David and Hildebrandt, Klaus and Technion, Mirela Ben-Chen and Panozzo, Daniele},
title = {Directional field synthesis, design, and processing},
year = {2017},
isbn = {9781450350143},
publisher = {Association for Computing Machinery},
address = {New York, NY, USA},
url = {https://doi.org/10.1145/3084873.3084921},
doi = {10.1145/3084873.3084921},
booktitle = {ACM SIGGRAPH 2017 Courses},
articleno = {12},
numpages = {30},
location = {Los Angeles, California},
series = {SIGGRAPH '17}
}

@article{crane2010,
author={Keenan Crane and Mathieu Desbrun and Peter Schr\"{o}der},
title={Trivial Connections on Discrete Surfaces},
journal={Computer Graphics Forum (SGP)},
volume={29},
number={5},
pages={1525-1533},
year={2010}
}

@article{benchen2010killing,
author = {Ben-Chen, Mirela and Butscher, Adrian and Solomon, Justin and Guibas, Leonidas},
title = {On Discrete Killing Vector Fields and Patterns on Surfaces},
journal = {Computer Graphics Forum},
volume = {29},
number = {5},
pages = {1701-1711},
doi = {https://doi.org/10.1111/j.1467-8659.2010.01779.x},
url = {https://onlinelibrary.wiley.com/doi/abs/10.1111/j.1467-8659.2010.01779.x},
eprint = {https://onlinelibrary.wiley.com/doi/pdf/10.1111/j.1467-8659.2010.01779.x},
year = {2010}
}

@article{falorsi2020,
  title={Neural Ordinary Differential Equations on Manifolds},
  author={Luca Falorsi and Patrick Forr'e},
  journal={ArXiv},
  year={2020},
  volume={abs/2006.06663},
  url={https://api.semanticscholar.org/CorpusID:219573582}
}

@inproceedings{lou2020,
 author = {Lou, Aaron and Lim, Derek and Katsman, Isay and Huang, Leo and Jiang, Qingxuan and Lim, Ser Nam and De Sa, Christopher M},
 booktitle = {Advances in Neural Information Processing Systems},
 editor = {H. Larochelle and M. Ranzato and R. Hadsell and M.F. Balcan and H. Lin},
 pages = {17548--17558},
 publisher = {Curran Associates, Inc.},
 title = {Neural Manifold Ordinary Differential Equations},
 url = {https://proceedings.neurips.cc/paper_files/paper/2020/file/cbf8710b43df3f2c1553e649403426df-Paper.pdf},
 volume = {33},
 year = {2020}
}

@inproceedings{mathieu2020,
author = {Mathieu, Emile and Nickel, Maximilian},
title = {Riemannian continuous normalizing flows},
year = {2020},
isbn = {9781713829546},
publisher = {Curran Associates Inc.},
address = {Red Hook, NY, USA},
booktitle = {Proceedings of the 34th International Conference on Neural Information Processing Systems},
articleno = {211},
numpages = {13},
location = {Vancouver, BC, Canada},
series = {NIPS '20}
}

@inproceedings{rozen2021,
 author = {Rozen, Noam and Grover, Aditya and Nickel, Maximilian and Lipman, Yaron},
 booktitle = {Advances in Neural Information Processing Systems},
 editor = {M. Ranzato and A. Beygelzimer and Y. Dauphin and P.S. Liang and J. Wortman Vaughan},
 pages = {17669--17680},
 publisher = {Curran Associates, Inc.},
 title = {Moser Flow: Divergence-based Generative Modeling on Manifolds},
 url = {https://proceedings.neurips.cc/paper_files/paper/2021/file/93a27b0bd99bac3e68a440b48aa421ab-Paper.pdf},
 volume = {34},
 year = {2021}
}

@inproceedings{bose2020,
  title = 	 {Latent Variable Modelling with Hyperbolic Normalizing Flows},
  author =       {Bose, Joey and Smofsky, Ariella and Liao, Renjie and Panangaden, Prakash and Hamilton, Will},
  booktitle = 	 {Proceedings of the 37th International Conference on Machine Learning},
  pages = 	 {1045--1055},
  year = 	 {2020},
  editor = 	 {III, Hal Daumé and Singh, Aarti},
  volume = 	 {119},
  series = 	 {Proceedings of Machine Learning Research},
  month = 	 {13--18 Jul},
  publisher =    {PMLR},
  url = 	 {https://proceedings.mlr.press/v119/bose20a.html}
}

@inproceedings{chen2024,
 author = {Chen, Ricky T. Q. and Lipman, Yaron},
 booktitle = {International Conference on Learning Representations},
 editor = {B. Kim and Y. Yue and S. Chaudhuri and K. Fragkiadaki and M. Khan and Y. Sun},
 pages = {47922--47945},
 title = {Flow Matching on General Geometries},
 url = {https://proceedings.iclr.cc/paper_files/paper/2024/file/d1f9936d3be6997ffffab692977eebe6-Paper-Conference.pdf},
 volume = {2024},
 year = {2024}
}

@inproceedings{benhamu2022,
  title = 	 {Matching Normalizing Flows and Probability Paths on Manifolds},
  author =       {Ben-Hamu, Heli and Cohen, Samuel and Bose, Joey and Amos, Brandon and Nickel, Maximillian and Grover, Aditya and Chen, Ricky T. Q. and Lipman, Yaron},
  booktitle = 	 {Proceedings of the 39th International Conference on Machine Learning},
  pages = 	 {1749--1763},
  year = 	 {2022},
  editor = 	 {Chaudhuri, Kamalika and Jegelka, Stefanie and Song, Le and Szepesvari, Csaba and Niu, Gang and Sabato, Sivan},
  volume = 	 {162},
  series = 	 {Proceedings of Machine Learning Research},
  month = 	 {17--23 Jul},
  publisher =    {PMLR},
  url = 	 {https://proceedings.mlr.press/v162/ben-hamu22a.html}
}

@article{campen2016,
author = {Campen, Marcel and Silva, Cl\'{a}udio T. and Zorin, Denis},
title = {Bijective maps from simplicial foliations},
year = {2016},
issue_date = {July 2016},
publisher = {Association for Computing Machinery},
address = {New York, NY, USA},
volume = {35},
number = {4},
issn = {0730-0301},
url = {https://doi.org/10.1145/2897824.2925890},
doi = {10.1145/2897824.2925890},
journal = {ACM Trans. Graph.},
month = jul,
articleno = {74},
numpages = {15}
}

@article{hinderink2021,
author = {Hinderink, Steffen and Campen, Marcel},
title = {Galaxy Maps: Localized Foliations for Bijective Volumetric Mapping},
year = {2023},
issue_date = {August 2023},
publisher = {Association for Computing Machinery},
address = {New York, NY, USA},
volume = {42},
number = {4},
issn = {0730-0301},
url = {https://doi.org/10.1145/3592410},
doi = {10.1145/3592410},
journal = {ACM Trans. Graph.},
month = jul,
articleno = {129},
numpages = {16}
}

@inproceedings{wang2025power,
  title={Power-Linear Polar Directional Fields},
  author={Wang, Jiabao Brad and Vaxman, Amir},
  booktitle={Proceedings of the Special Interest Group on Computer Graphics and Interactive Techniques Conference Conference Papers},
  pages={1--10},
  year={2025}
}

@article{chen2018neural,
  title={Neural ordinary differential equations},
  author={Chen, Ricky TQ and Rubanova, Yulia and Bettencourt, Jesse and Duvenaud, David K},
  journal={Advances in neural information processing systems},
  volume={31},
  year={2018}
}

@article{schmidt2020inter,
  title={Inter-surface maps via constant-curvature metrics},
  author={Schmidt, Patrick and Campen, Marcel and Born, Janis and Kobbelt, Leif},
  journal={ACM Transactions on Graphics (TOG)},
  volume={39},
  number={4},
  pages={119--1},
  year={2020},
  publisher={ACM New York, NY, USA}
}

@article{aigerman2015seamless,
  title={Seamless surface mappings},
  author={Aigerman, Noam and Poranne, Roi and Lipman, Yaron},
  journal={ACM Transactions on Graphics (TOG)},
  volume={34},
  number={4},
  pages={1--13},
  year={2015},
  publisher={ACM New York, NY, USA}
}

@article{aigerman2014lifted,
  title={Lifted bijections for low distortion surface mappings},
  author={Aigerman, Noam and Poranne, Roi and Lipman, Yaron},
  journal={ACM Transactions on Graphics (TOG)},
  volume={33},
  number={4},
  pages={1--12},
  year={2014},
  publisher={ACM New York, NY, USA}
}

@article{aigerman2016hyperbolic,
  title={Hyperbolic orbifold tutte embeddings.},
  author={Aigerman, Noam and Lipman, Yaron},
  journal={ACM Trans. Graph.},
  volume={35},
  number={6},
  pages={217--1},
  year={2016}
}

@incollection{schreiner2004inter,
  title={Inter-surface mapping},
  author={Schreiner, John and Asirvatham, Arul and Praun, Emil and Hoppe, Hugues},
  booktitle={ACM SIGGRAPH 2004 Papers},
  pages={870--877},
  year={2004}
}

@article{aigerman2015orbifold,
  title={Orbifold tutte embeddings.},
  author={Aigerman, Noam and Lipman, Yaron},
  journal={ACM Trans. Graph.},
  volume={34},
  number={6},
  pages={190--1},
  year={2015}
}

@inproceedings{tsui2013globally,
  title={Globally optimal cortical surface matching with exact landmark correspondence},
  author={Tsui, Alex and Fenton, Devin and Vuong, Phong and Hass, Joel and Koehl, Patrice and Amenta, Nina and Coeurjolly, David and DeCarli, Charles and Carmichael, Owen},
  booktitle={International Conference on Information Processing in Medical Imaging},
  pages={487--498},
  year={2013},
  organization={Springer}
}

@article{aigerman2017spherical,
  title={Spherical orbifold tu e embeddings},
  author={Aigerman, Noam and Kovalsky, Shahar Z and Lipman, Yaron},
  journal={ACM Trans. Graph},
  volume={36},
  number={4},
  pages={90},
  year={2017}
}

@article{boksebeld2022high,
  title={High-order directional fields},
  author={Boksebeld, Iwan and Vaxman, Amir},
  journal={ACM Transactions on Graphics (TOG)},
  volume={41},
  number={6},
  pages={1--17},
  year={2022},
  publisher={ACM New York, NY, USA}
}

@article{solomon2019optimal,
  title={Optimal transport-based polar interpolation of directional fields},
  author={Solomon, Justin and Vaxman, Amir},
  journal={ACM Transactions on Graphics (TOG)},
  volume={38},
  number={4},
  pages={1--13},
  year={2019},
  publisher={ACM New York, NY, USA}
}

@article{huynh2015,
  title={Bijective Deformations in $\mathbb{R}^n$ via Integral Curve Coordinates},
  author={Huynh, Lisa and Gingold, Yotam},
  journal={arXiv preprint arXiv:1505.00073},
  year={2015}
}

@article{nieser2011,
author = {Nieser, M. and Reitebuch, U. and Polthier, K.},
title = {CubeCover– Parameterization of 3D Volumes},
journal = {Computer Graphics Forum},
volume = {30},
number = {5},
pages = {1397-1406},
doi = {https://doi.org/10.1111/j.1467-8659.2011.02014.x},
url = {https://onlinelibrary.wiley.com/doi/abs/10.1111/j.1467-8659.2011.02014.x},
eprint = {https://onlinelibrary.wiley.com/doi/pdf/10.1111/j.1467-8659.2011.02014.x},
year = {2011}
}

@article{bommes2013,
author = {Bommes, David and Campen, Marcel and Ebke, Hans-Christian and Alliez, Pierre and Kobbelt, Leif},
title = {Integer-grid maps for reliable quad meshing},
year = {2013},
issue_date = {July 2013},
publisher = {Association for Computing Machinery},
address = {New York, NY, USA},
volume = {32},
number = {4},
issn = {0730-0301},
url = {https://doi.org/10.1145/2461912.2462014},
doi = {10.1145/2461912.2462014},
journal = {ACM Trans. Graph.},
month = jul,
articleno = {98},
numpages = {12}
}

@article{hinderink2024,
author = {Hinderink, Steffen and Br\"{u}ckler, Hendrik and Campen, Marcel},
title = {Bijective Volumetric Mapping via Star Decomposition},
year = {2024},
issue_date = {December 2024},
publisher = {Association for Computing Machinery},
address = {New York, NY, USA},
volume = {43},
number = {6},
issn = {0730-0301},
url = {https://doi.org/10.1145/3687950},
doi = {10.1145/3687950},
journal = {ACM Trans. Graph.},
month = nov,
articleno = {168},
numpages = {11}
}

@article{jiang2020,
author = {Jiang, Zhongshi and Schneider, Teseo and Zorin, Denis and Panozzo, Daniele},
title = {Bijective projection in a shell},
year = {2020},
issue_date = {December 2020},
publisher = {Association for Computing Machinery},
address = {New York, NY, USA},
volume = {39},
number = {6},
issn = {0730-0301},
url = {https://doi.org/10.1145/3414685.3417769},
doi = {10.1145/3414685.3417769},
journal = {ACM Trans. Graph.},
month = nov,
articleno = {247},
numpages = {18}
}

@article{jiang2021,
author = {Jiang, Zhongshi and Zhang, Ziyi and Hu, Yixin and Schneider, Teseo and Zorin, Denis and Panozzo, Daniele},
title = {Bijective and coarse high-order tetrahedral meshes},
year = {2021},
issue_date = {August 2021},
publisher = {Association for Computing Machinery},
address = {New York, NY, USA},
volume = {40},
number = {4},
issn = {0730-0301},
url = {https://doi.org/10.1145/3450626.3459840},
doi = {10.1145/3450626.3459840},
journal = {ACM Trans. Graph.},
month = jul,
articleno = {157},
numpages = {16}
}

@article{liu2024,
author = {Liu, Shibo and Ji, Yang and Guo, Jia-Peng and Liu, Ligang and Fu, Xiao-Ming},
title = {Smooth Bijective Projection in a High-order Shell},
year = {2024},
issue_date = {July 2024},
publisher = {Association for Computing Machinery},
address = {New York, NY, USA},
volume = {43},
number = {4},
issn = {0730-0301},
url = {https://doi.org/10.1145/3658207},
doi = {10.1145/3658207},
journal = {ACM Trans. Graph.},
month = jul,
articleno = {59},
numpages = {13}
}

@article{zhang2006vector,
  title={Vector field design on surfaces},
  author={Zhang, Eugene and Mischaikow, Konstantin and Turk, Greg},
  journal={ACM Transactions on Graphics (ToG)},
  volume={25},
  number={4},
  pages={1294--1326},
  year={2006},
  publisher={ACM New York, NY, USA}
}

@article{knoppel2013globally,
  title={Globally optimal direction fields},
  author={Kn{\"o}ppel, Felix and Crane, Keenan and Pinkall, Ulrich and Schr{\"o}der, Peter},
  journal={ACM Transactions on Graphics (ToG)},
  volume={32},
  number={4},
  pages={1--10},
  year={2013},
  publisher={ACM New York, NY, USA}
}

@article{myles2014robust,
  title={Robust field-aligned global parametrization.},
  author={Myles, Ashish and Pietroni, Nico and Zorin, Denis},
  journal={ACM Trans. Graph.},
  volume={33},
  number={4},
  pages={135--1},
  year={2014}
}

@inproceedings{bommes2013quad,
  title={Quad-mesh generation and processing: A survey},
  author={Bommes, David and L{\'e}vy, Bruno and Pietroni, Nico and Puppo, Enrico and Silva, Claudio and Tarini, Marco and Zorin, Denis},
  booktitle={Computer graphics forum},
  volume={32},
  number={6},
  pages={51--76},
  year={2013},
  organization={Wiley Online Library}
}

@article{fisher2007design,
  title={Design of tangent vector fields},
  author={Fisher, Matthew and Schr{\"o}der, Peter and Desbrun, Mathieu and Hoppe, Hugues},
  journal={ACM transactions on graphics (TOG)},
  volume={26},
  number={3},
  pages={56--es},
  year={2007},
  publisher={ACM New York, NY, USA}
}

@article{ray2006periodic,
  title={Periodic global parameterization},
  author={Ray, Nicolas and Li, Wan Chiu and L{\'e}vy, Bruno and Sheffer, Alla and Alliez, Pierre},
  journal={ACM Transactions on Graphics (TOG)},
  volume={25},
  number={4},
  pages={1460--1485},
  year={2006},
  publisher={ACM New York, NY, USA}
}

@article{diamanti2015integrable,
  title={Integrable polyvector fields},
  author={Diamanti, Olga and Vaxman, Amir and Panozzo, Daniele and Sorkine-Hornung, Olga},
  journal={ACM Transactions on Graphics (TOG)},
  volume={34},
  number={4},
  pages={1--12},
  year={2015},
  publisher={ACM New York, NY, USA}
}

@article{gortler2006,
title = {Discrete one-forms on meshes and applications to 3D mesh parameterization},
journal = {Computer Aided Geometric Design},
volume = {23},
number = {2},
pages = {83-112},
year = {2006},
issn = {0167-8396},
doi = {https://doi.org/10.1016/j.cagd.2005.05.002},
url = {https://www.sciencedirect.com/science/article/pii/S0167839605000531},
author = {Steven J. Gortler and Craig Gotsman and Dylan Thurston}
}

@article{tutte1963,
author = {Tutte, W. T.},
title = {How to Draw a Graph},
journal = {Proceedings of the London Mathematical Society},
volume = {s3-13},
number = {1},
pages = {743-767},
doi = {https://doi.org/10.1112/plms/s3-13.1.743},
url = {https://londmathsoc.onlinelibrary.wiley.com/doi/abs/10.1112/plms/s3-13.1.743},
eprint = {https://londmathsoc.onlinelibrary.wiley.com/doi/pdf/10.1112/plms/s3-13.1.743},
year = {1963}
}

@article{floater1997parametrization,
  title={Parametrization and smooth approximation of surface triangulations},
  author={Floater, Michael S},
  journal={Computer aided geometric design},
  volume={14},
  number={3},
  pages={231--250},
  year={1997},
  publisher={Elsevier}
}

@article{floater2003one,
  title={One-to-one piecewise linear mappings over triangulations},
  author={Floater, Michael},
  journal={Mathematics of Computation},
  volume={72},
  number={242},
  pages={685--696},
  year={2003}
}

@inproceedings{sun2024tuttenet,
  title={TutteNet: Injective 3D Deformations by Composition of 2D Mesh Deformations},
  author={Sun, Bo and Groueix, Thibault and Song, Chen and Huang, Qixing and Aigerman, Noam},
  booktitle={Proceedings of the IEEE/CVF Conference on Computer Vision and Pattern Recognition},
  pages={21378--21389},
  year={2024}
}

@inproceedings{aigerman2024generative,
  title={Generative escher meshes},
  author={Aigerman, Noam and Groueix, Thibault},
  booktitle={ACM SIGGRAPH 2024 Conference Papers},
  pages={1--11},
  year={2024}
}

@article{jiang2020shapeflow,
  title={Shapeflow: Learnable deformation flows among 3d shapes},
  author={Jiang, Chiyu and Huang, Jingwei and Tagliasacchi, Andrea and Guibas, Leonidas J},
  journal={Advances in Neural Information Processing Systems},
  volume={33},
  pages={9745--9757},
  year={2020}
}

@article{panozzo2014frame,
  title={Frame fields: Anisotropic and non-orthogonal cross fields additional material},
  author={Panozzo, Daniele and Puppo, Enrico and Tarini, Marco and Sorkine-Hornung, Olga},
  journal={Proceedings of the ACM TRANSACTIONS ON GRAPHICS (PROCEEDINGS OF ACM SIGGRAPH},
  volume={3},
  year={2014}
}

@inproceedings{bommes2010practical,
  title={Practical mixed-integer optimization for geometry processing},
  author={Bommes, David and Zimmer, Henrik and Kobbelt, Leif},
  booktitle={International conference on curves and surfaces},
  pages={193--206},
  year={2010},
  organization={Springer}
}

@article{palacios2007rotational,
  title={Rotational symmetry field design on surfaces},
  author={Palacios, Jonathan and Zhang, Eugene},
  journal={ACM Transactions on Graphics (TOG)},
  volume={26},
  number={3},
  pages={55--es},
  year={2007},
  publisher={ACM New York, NY, USA}
}

@inproceedings{diamanti2014designing,
  title={Designing N-PolyVector fields with complex polynomials},
  author={Diamanti, Olga and Vaxman, Amir and Panozzo, Daniele and Sorkine-Hornung, Olga},
  booktitle={Computer Graphics Forum},
  volume={33},
  number={5},
  pages={1--11},
  year={2014},
  organization={Wiley Online Library}
}

@inproceedings{xia2010direct,
  title={Direct-product volumetric parameterization of handlebodies via harmonic fields},
  author={Xia, Jiazhi and He, Ying and Yin, Xiaotian and Han, Shuchu and Gu, Xianfeng},
  booktitle={2010 Shape Modeling International Conference},
  pages={3--12},
  year={2010},
  organization={IEEE}
}

@article{lin2015interior,
  title={Interior structure transfer via harmonic 1-forms},
  author={Lin, Juncong and Xia, Jiazhi and Gao, Xing and Liao, Minghong and He, Ying and Gu, Xianfeng},
  journal={Multimedia Tools and Applications},
  volume={74},
  number={1},
  pages={139--158},
  year={2015},
  publisher={Springer}
}

@inproceedings{hefetz2019subspace,
  title={A subspace method for fast locally injective harmonic mapping},
  author={Hefetz, Eden Fedida and Chien, Edward and Weber, Ofir},
  booktitle={Computer Graphics Forum},
  volume={38},
  number={2},
  pages={105--119},
  year={2019},
  organization={Wiley Online Library}
}

@article{bright2017harmonic,
  title={Harmonic global parametrization with rational holonomy},
  author={Bright, Alon and Chien, Edward and Weber, Ofir},
  journal={ACM Transactions on Graphics (TOG)},
  volume={36},
  number={4},
  pages={1--15},
  year={2017},
  publisher={ACM New York, NY, USA}
}

@article{pinkall1993,
author = {Ulrich Pinkall and Konrad Polthier},
title = {Computing Discrete Minimal Surfaces and Their Conjugates},
journal = {Experimental Mathematics},
volume = {2},
number = {1},
pages = {15--36},
year = {1993},
publisher = {Taylor \& Francis},
doi = {10.1080/10586458.1993.10504266},
url = {https://doi.org/10.1080/10586458.1993.10504266},
eprint = {https://doi.org/10.1080/10586458.1993.10504266},
}

@inproceedings{kingma2015,
  author       = {Diederik P. Kingma and
                  Jimmy Ba},
  editor       = {Yoshua Bengio and
                  Yann LeCun},
  title        = {Adam: {A} Method for Stochastic Optimization},
  booktitle    = {3rd International Conference on Learning Representations, {ICLR} 2015,
                  San Diego, CA, USA, May 7-9, 2015, Conference Track Proceedings},
  year         = {2015},
  url          = {http://arxiv.org/abs/1412.6980},
  bibsource    = {dblp computer science bibliography, https://dblp.org}
}

@article{floater2003mean,
  title={Mean value coordinates},
  author={Floater, Michael S},
  journal={Computer aided geometric design},
  volume={20},
  number={1},
  pages={19--27},
  year={2003},
  publisher={Elsevier}
}

@inproceedings{kraevoy2004,
author = {Kraevoy, Vladislav and Sheffer, Alla},
title = {Cross-parameterization and compatible remeshing of 3D models},
year = {2004},
isbn = {9781450378239},
publisher = {Association for Computing Machinery},
address = {New York, NY, USA},
url = {https://doi.org/10.1145/1186562.1015811},
doi = {10.1145/1186562.1015811},
booktitle = {ACM SIGGRAPH 2004 Papers},
pages = {861–869},
numpages = {9},
location = {Los Angeles, California},
series = {SIGGRAPH '04}
}

@InProceedings{morreale2021,
title={Neural Surface Maps},
author={Morreale, Luca and Aigerman, Noam and Kim, Vladimir G and Mitra, Niloy J},
booktitle={Proceedings of the IEEE/CVF Conference on Computer Vision and Pattern Recognition},
pages={4639--4648},
year={2021}
}
\appendix
\section{Proofs of Lemmas}
\label{app:lemma}
\begin{proof}[Proof of Lemma~\ref{lem:flux_1form_equiv}]
    Let $\e=(i,j)$ be a half-edge, let $\tri=(i,j,k)$ be its triangle, and define the edge vector as $\dir=\frac{v_j-v_i}{\norm{v_i-v_j}}$. In standard linear FEM fashion, consider $\z$ as defining an affine function over the triangle, $z(p)=a^Tp+b$ for any point $p\in\tri$, uniquely defined as requiring that it interpolates the values $\z_i,\z_j,\z_k$ on the vertices, i.e., $z(v_\alpha)=\z_\alpha, \alpha\in\set{i,j,k}$. We will prove that the sign of the directional derivative $\frac{\partial z}{\partial \dir}$ in the direction $\dir$ is equal to both terms in the lemma. First, by restricting $z$ to a one-dimensional function over the edge, we have $\frac{\partial z}{\partial \dir}=\frac{\z_j-\z_i}{\norm{v_i-v_j}}\overset{\eqref{eq:1form}}{=}\frac{\delta_\e}{\norm{v_i-v_j}}$; therefore, $\sign{\frac{\partial z}{\partial \dir}} =\sign{\delta_\e}$.
    Second, a directional derivative is also equal to the inner product between the gradient and the direction, i.e.,  $\frac{\partial z}{\partial \dir}=\tribrac{\nabla_\tri z,\dir}$. Since applying the same rotation to both vectors does not  change their inner product, we have $\frac{\partial z}{\partial \dir}=\tribrac{\textbf{R}^{90^\circ}\nabla_\tri z,\textbf{R}^{90^\circ}\dir}=\tribrac{\vv_\tri,-\n_\e}$. Thus, we get $\sign{\frac{\partial z}{\partial \dir}} = -\sign{\tribrac{\vv_\tri,\n_\e}} $. 
    Hence, $\sign{\delta_\e} = -\sign{\tribrac{\vv_\tri,\n_\e}}=\flux_\e$.
\end{proof}
\begin{proof}[Proof of Lemma~\ref{lem:us_gortler_equiv}]
We first convert the vertex flux to a more relaxed definition (refer also to Figure~\ref{fig:vertex_inj}):
\begin{definition}
    {Flux parity:} for triangle $(o,i,j)$, we say the triangle  has \emph{even}-parity flux with respect to the vertex $o$ if $\flux_{o,i} = \flux_{j,o}$; otherwise, if $\flux_{o,i} = - \flux_{j,o}$, we say the triangle has \emph{odd}-parity flux  with respect to the vertex $o$. 
\end{definition}
In other words, positive/negative vertex fluxes are even-parity fluxes, and a zero vertex flux is odd-parity (see Figure~\ref{fig:vertex_inj} for an illustration, where triangles are highlighted in red if they have even-parity flux with respect to the 1-ring's center vertex). Hence, translating the conditions of Property~\ref{prop:vert_inj} to parity, it is a necessary condition that two triangles have even-parity flux and the rest have odd-parity flux; the next lemma shows that in fact these two conditions are exactly equivalent, i.e., the parity condition ensures vertex injectivity:

\begin{lemma}
\label{lemma:v_inj}
    i) The flux is an injective \textit{regular} flux over a vertex $o$ iff there are exactly two triangles $\tri^1,\tri^2$ of the vertex's 1-ring that have \textit{even}-parity vertex flux (with respect to $o$) and all the rest of the triangles have odd-parity vertex flux. ii) The flux is an injective \textit{vortex} flux iff all triangles have \textit{odd}-parity vertex flux (with respect to $o$). 
\end{lemma}

We prove this lemma later.
Now, as illustrated in Figure~\ref{fig:vertex_inj}, note that two consecutive edges $(o,i),(o,i+1)$ of the 1-ring correspond to the oriented half-edges $(o,i),(i+1,o)$ of the triangle $(o,i,i+1)$. Hence, if there's no sign change ($\flux_{(o,i)}=\flux_{(o,i+1)}$), then by the anti-symmetry of the 1-form we have $\flux_{(o,i)}=-\flux_{(i+1,o)}$, i.e., the triangle has odd-parity flux into $o$. Similarly, if there is a sign change ($\flux_{(o,i)}=-\flux_{(o,i+1)}$), then by the anti-symmetry of the 1-form we have $\flux_{(o,i)}=\flux_{(i+1,o)}$ and the triangle has even-parity flux into $o$. 

In sum: an index of $1$ is defined as no sign changes in the 1-ring, in turn meaning all triangles in the 1-ring have odd-parity flux, which by Lemma~\ref{lemma:v_inj} entails a vortex injective vertex, as defined in Subsection~\ref{ss:injective_flux} (and vice versa: having all odd-parity flux triangles entails an index of $1$). 

Similarly, an index of $0$ is defined as exactly two sign changes, hence exactly two triangles of the 1-ring  have even-parity flux, which per Lemma~\ref{lemma:v_inj} entails a regular injective vertex, as defined in Subsection~\ref{ss:injective_flux} (and vice versa: two even-parity triangles entail an index of $0$). 

All other numbers of sign changes have negative index, and all correspond to the last remaining case, of a non-injective vertex. With this, we have proven Lemma~\ref{lem:us_gortler_equiv}. \end{proof}

\begin{proof}[Proof of Lemma~\ref{lem:non_neg_index}]
We begin by reducing the equation in Theorem 3.5 in~\cite{gortler2006} to our setting. The original equation stipulates that for a polygonal mesh of genus $g$ with a given 1-form $\delta$, 
\begin{equation}
\label{eq:gortler}
    \sum_{v \in \textbf{V}} \underbrace{\ind{v}}_{\text{vertex index}} + \sum_{\tri \in \textbf{T}} \underbrace{\ind{\tri}}_{\text{face index}} = 2 - 2g.
\end{equation}
Aside from the vertex index term (defined in Subsection~\ref{ss:injectivity_proof}), there is an additional face index term; however, it is irrelevant to us: this is computed for \emph{polygonal} meshes, by walking over the (polygonal) face's edges and counting how many sign changes occur; for our case this term is always zero since, by definition of $\delta_i$ (Equation~\eqref{eq:1form}), we have that for a triangle with edges $\e_1,\e_2,\e_3$ it holds that $\delta_1+\delta_2+\delta_3 =0$,  entailing exactly two sign changes, thereby leading to zero contribution. Plugging in the genus of the sphere ($g=0$) to the right-hand side of Equation~\eqref{eq:gortler}, we get from their theorem that our 1-form $\delta$ satisfies
\begin{equation}
    \sum_{v\in\mathcal{V}} \text{ind}\parr{v} = 2.
\end{equation}
We note that all vertices except the source and sink satisfy the discrete harmonic equation, and thus it is well-known (and easy to prove) that they satisfy a discrete \textit{maximum principle}: a vertex $o$ has a vertex $v_j$ in its 1-ring with a higher value $\z_o<\z_j$, and a vertex $v_k$ with a lower value $\z_o>\z_k$. By the definition of $\delta$, Equation~\eqref{eq:1form}, this entails $\delta_{(o,j)}>0$, and $\delta_{(o,k)}<0$. As such, there must be at least two sign changes transpiring across the edges of the 1-ring of $o$. Thus  the index of all vertices is non-positive, except the source and sink. Since the indices sum to $2$, this implies both the source and the sink must have the maximal index, i.e., $1$. With that, the sum of the indices of all the other vertices must be exactly zero. Due to non-positivity, this entails that the index of each of those vertices is  zero. \end{proof}

\begin{proof}[Proof of Lemma~\ref{lemma:v_inj}]
($\Rightarrow$) i) Assume the vertex flux is regular injective, i.e., per the definition in Section~\ref{ss:injective_flux}, Property~\ref{prop:vert_inj}, there is one triangle $\tri^+$ with positive vertex flux, and one $\tri^-$ with negative vertex flux, and all others with zero vertex flux. Then, by the definition of parity, both $\tri^+,\tri^-$ have even-parity vertex flux while the rest have odd-parity vertex flux, as desired. ii) Assume the vertex flux is vortex injective. Then all triangles have zero vertex flux, which has odd parity.

($\Leftarrow$) i) Assume that  there are exactly two  triangles $\tri^1,\tri^2$ with even-parity vertex flux, with the rest having odd-parity vertex flux. We need to show that the definition of regular injective vertex from Subsection~\ref{ss:injective_flux}, Property~\ref{prop:vert_inj} is satisfied. Odd-parity flux triangles have zero vertex flux; hence, all that remains is to prove that one of $\tri^1,\tri^2$ has positive vertex flux, and one has negative vertex flux (as opposed to both positive or both negative). Without loss of generality, assume $\tri^1$ has positive vertex flux. Due to Equation~\eqref{eq:vertex_flux}, we have $\flux_{(o,i)}=\flux_{(i+1,o)}=+1$. Using the anti-symmetry between a half-edge and its twin (Equation~\eqref{eq:1form}), we have $\flux_{i+1,o}=-\flux_{o,i+1}$, entailing $\flux_{o,i+1}=-1$. 
We consider two cases:
\begin{enumerate}
    \item If $(o,i+1)$ is a half-edge of $\tri^2$, we are done: since $\tri^2$ has even-parity flux, its other edge satisfies $\flux_{\parr{i+2,o}}=\flux_{\parr{o,i+1}}=-1$, and thus both have negative edge  flux, entailing  the vertex flux is negative.
    \item If the half-edge belongs to a triangle other than  $\tri^2$, it is a triangle with odd-parity flux, implying $\flux_{(i+2,o)}=-\flux_{(o,i+1)}=-(-1)=+1$. By induction, we can show that when iterating over all triangles until reaching $\tri^2$, for each triangle $(o,j,j+1)$ we have  $\flux_{(j,o)}=1$ and $\flux_{(o,j+1)}=-1$. Once we reach $\tri^2$, we are once again in the setting of case 1, and the proof is concluded. 
\end{enumerate}
ii) Assume all triangles have odd-parity vertex flux; then there are no sign changes, and each triangle must have zero vertex flux, i.e., entailing the vertex flux is vortex injective.
\end{proof}

\newpage

\end{document}